\documentclass[letter]{amsart}

\usepackage[utf8]{inputenc}
\usepackage[english]{babel}
\usepackage{amsmath,amssymb,amsthm,mathrsfs}
\usepackage[colorlinks,linktocpage]{hyperref}
\usepackage[all]{xy}
\usepackage{import}
\usepackage{pdfpages}
\usepackage{transparent}

\usepackage{outlines}
\usepackage{mdframed}
\usepackage{enumitem}

\def\dim{\mathrm{dim}}
\def\Ric{\mathrm{Ric}}
\def\Vol{\mathrm{Vol}}

\def\Id{\mathrm{Id}}

\newtheorem{thm}{Theorem}[section]
\newtheorem{lem}[thm]{Lemma}

\newtheorem{prop}[thm]{Proposition}
\newtheorem{conj}[thm]{Conjecture}
\newtheorem{ques}[thm]{Question}

\newtheorem*{claim*}{Claim}

\theoremstyle{definition}
\newtheorem{defn}{Definition}[section]

\theoremstyle{remark}
\newtheorem{rmk}{Remark}[section]
\newtheorem*{rmk*}{Remark}

\newtheorem*{fact*}{Fact}

\title[Some stability results of the positive mass theorem]{Some stability results of positive mass theorem for uniformly asymptotically flat $3$-manifolds}
\author{Conghan Dong}
\address{Mathematics Department, Stony Brook University, NY 11794, United States}
\email{conghan.dong@stonybrook.edu}

\begin{document}
%\date{\today}

\maketitle
\begin{abstract}
	In this paper, we show that for a sequence of orientable complete uniformly asymptotically flat $3$-manifolds $(M_i, g_i)$ with nonnegative scalar curvature and ADM mass $m(g_i)$ tending to zero, by subtracting some open subsets $Z_i$, whose boundary area satisfies $\mathrm{Area}(\partial Z_i) \leq C m(g_i)^{\frac{1}{2}- \varepsilon }$, for any base point $p_i \in M_i\setminus  Z_i$, $(M_i\setminus Z_i, g_i, p_i)$ converges to the Euclidean space $(\mathbb{R}^3, g_E, 0)$ in the  $C^0$ modulo negligible volume sense. Moreover, if we assume that the Ricci curvature is uniformly bounded from below, then $(M_i, g_i, p_i)$ converges to $(\mathbb{R}^3, g_E, 0)$ in the pointed Gromov-Hausdorff topology. 

\medskip

\noindent\textsc{R\'esum\'e.}
Dans cet article, nous démontrons que pour une suite de variétés tridimensionnelles orientables, complètes et uniformément asymptotiquement plates $(M_i, g_i)$ avec courbure scalaire non négative et dont la masse ADM $m(g_i)$ tend vers zéro, en enlevant certains sous-ensembles ouverts $Z_i$, dont la zone de frontière satisfait $\mathrm{Aire}(\partial Z_i) \leq C \cdot m(g_i)^{\frac{1}{2} - \varepsilon}$, pour tout point de base $p_i \in M_i \setminus Z_i$, le triplet $(M_i \setminus Z_i, g_i, p_i)$ converge vers l'espace euclidien $(\mathbb{R}^3, g_E, 0)$ dans le sens $C^0$, modulo un volume négligeable. De plus, si nous supposons que la courbure de Ricci est uniformément bornée inférieurement, alors le triplet $(M_i, g_i, p_i)$ converge vers $(\mathbb{R}^3, g_E, 0)$ pour la topologie de Gromov-Hausdorff pointée.

\end{abstract}

\tableofcontents

\section{Introduction}
A smooth orientable connected complete Riemannian $3$-manifold $(M^3, g)$ with one end is called $(A, B, \sigma )$-asymptotically flat (or $(A,B, \sigma )$-AF for short) for some given $A,B>0$ and $ \sigma >\frac{1}{2}$, if there exists a compact subset $K \subset M$ and a $C^\infty$-diffeomorphism $\Phi : M \setminus K \to  \mathbb{R}^3 \setminus B(0, A)$ such that under this identification, 
	$$
	| \partial ^{l}(g_{ij}- \delta _{ij})(x)| \leq B|x|^{-\sigma - |l|},\ \ \forall x \in \mathbb{R}^3\setminus B(0,A)
	$$ for all multi-indices $|l|=0,1,2$. Furthermore, we always assume the scalar curvature $R_g$ is integrable. 

	In general, a connected $3$-manifold $(M^3,g)$ with more than one end is called an AF $3$-manifold if there exists a compact subset $K \subset M$ such that $M\setminus K$ consists of finite pairwise disjoint ends $\{ M_{end}^k\}_{k=1}^N$, and each end $M_{end}^k$ is $(A_k, B_k, \sigma_k )$-AF as defined above for some $A_k, B_k>0$ and $\sigma_k > \frac{1}{2}$.

	Given an AF $3$-manifold $(M^3, g)$ with ends $M_{end}^{1},\ldots, M_{end}^N$, the ADM mass of the end $M_{end}^k$, coming from general relativity \cite{ADM61}, is defined as 
	$$
	m(M_{end}^k, g)=\lim_{r\to \infty}\frac{1}{16\pi}\int_{S_r}\sum_{i,j=1}^3(\partial _ig_{ij}-\partial _j g_{ii})\nu ^j dA,
$$ where all quantities in the integral are computed using the Euclidean background metric determined by the asymptotically flat coordinate in $M_{end}^k$ and $\nu $ is the unit normal vector to the standard sphere $S_r \subset \mathbb{R}^3$ with radius $r$. It has been shown in \cite{Bartnik86} that the mass of each end is finite and independent of the asymptotically flat coordinates. 

Let $m(g)$ be the maximum of $\{m(M_{end}^k, g)\}_{k=1}^N$, and $(\mathbb{R}^3, g_E)$ be the standard Euclidean $3$-space. We have the following positive mass theorem.
	\begin{thm}
		Let $(M^3, g)$ be an AF $3$-manifold with nonnegative scalar curvature. Then $m(g) \geq 0$, and equality holds if and only if $(M,g)=(\mathbb{R}^3, g_E)$ isometrically.
	\end{thm}
	This theorem was first proved by Schoen-Yau \cite{SchoenYau79a} in 1979 through the construction of stable minimal surfaces. Later, Schoen-Yau extended the theorem to dimensions less than eight \cite{Schoen89, SchoenYau79b}. In 1981, Witten \cite{Witten81} further proved the positive mass theorem for spin manifolds of any dimension. In 2001, as a byproduct of the proof of the Penrose inequality, Huisken-Ilmanen \cite{HuiskenIlmanen01} established the three-dimensional positive mass theorem using the inverse mean curvature flow. More recently, Li \cite{Li18} provided a proof of the three-dimensional positive mass theorem using Ricci flow, and Bray-Kazaras-Khuri-Stern \cite{BKKS22} gave a proof using a mass inequality.

	The positive mass theorem provides a rigidity statement. Consequently, a natural question arises regarding the stability problem.
\begin{ques}\label{question}
	Does the smallness of the mass imply that the manifold is close to the Euclidean space in some topology?
\end{ques}

It is currently unknown which topology would be a suitable choice for studying such problems in the context of nonnegative scalar curvature. In 2001, Huisken-Ilmanen \cite{HuiskenIlmanen01} proposed a conjecture regarding this question, specifically in terms of the Gromov-Hausdorff topology.

\begin{conj}\label{HI'conj}
	Suppose $M_i$ is a sequence of asymptotically flat $3$-manifolds with nonnegative scalar curvature and ADM mass tending to zero. Then there exists a subset $Z_i \subset M_i$ such that the boundary area $\mathrm{Area}(\partial Z_i)$ approaches zero, and $M_i \setminus Z_i$ converges to $\mathbb{R}^3$ in the Gromov-Hausdorff topology.
\end{conj}

\begin{rmk}
	Huisken-Ilmanen's motivation of this conjecture stems from the Riemannian Penrose inequality (see Section \ref{AE-harmonic-coord} for more details). Ilmanen also conjectured that $Z_i$ can be chosen in such a way that the boundary area satisfies $\mathrm{Area}(\partial Z_i) \leq 16 \pi m(g)^2$.
\end{rmk}

In 2014, Lee-Sormani \cite{LS14} proposed a conjecture related to the intrinsic flat metric topology (also mentioned in \cite[Conjecture 10.1]{Sormani23}). 
Recently, Lee-Naber-Neumayer \cite{LNN20} introduced $d_p$-convergence, and also formulated a conjecture related to this topology. In addition, significant progress has been made toward addressing the aforementioned question and conjectures, considering additional conditions. For more details, please refer to \cite{BF99, FK02, Cor05, Lee09, LS14, HL15, KKL21, ABK22}.

In this paper, we investigate the stability of the positive mass theorem without imposing additional conditions, inspired by Conjectures \ref{HI'conj} and the recent work of Kazaras-Khuri-Lee \cite{KKL21}. 

Similar to Cheeger-Gromov's smooth convergence, we give the following definition. 

\begin{defn}\label{C^0 modulo}
	Given a sequence of smooth Riemannian manifolds $(\Omega _i, g_i)$ and $(\Omega ,g)$, we say that $(\Omega _i, g_i)$ converges to $(\Omega ,g)$ in the ``$C^0$ modulo negligible volume'' sense if there exist subsets  $U_i \subset \Omega _i$ and embedding maps $\varphi _i: U_i \to \Omega $ such that $\mathrm{Vol}_{g_i}(\Omega _i \setminus U_i) \to 0$, $\mathrm{Vol}_g(\Omega  \setminus \varphi _i(U_i) ) \to 0$, and $| (\varphi _i)_* g_i - g|_{C^0( \varphi _i(U_i) )} \to 0$.

	Given a sequence of pointed smooth Riemannian manifolds  $(M_i, g_i, p_i)$ and $(M, g, p)$, possibly with boundary, we say that $(M_i, g_i, p_i)$ converges to $(M,g, p)$ in the ``pointed $C^0$ modulo negligible volume'' sense if, for any fixed radius $\rho $, there exist subsets $\Omega _i \supset \hat{B}(p_i, \rho )$ and $\Omega \supset \hat{B}(p, \rho )$ such that $(\Omega _i, g_i)$ converges to $(\Omega ,g)$ in the ``$C^0$ modulo negligible volume'' sense. Here, $\hat{B}(p, \rho )$ is the geodesic ball with respect to the induced length metric in $M$, i.e. $\hat{B}(p, \rho )= \{ x \in M: \inf \{\mathrm{Length}_g(\gamma ): \gamma (0)=p, \gamma (1)=x, \gamma \subset M\} \leq \rho  \} $. 
\end{defn}

Our main result gives one answer to Question \ref{question} in the sense of this pointed $C^0$-convergence modulo negligible volume. 

To make the statement precise, we introduce some notations. Consider an AF $3$-manifold $(M,g)$ with a diffeomorphism $\Phi $ defined on one end. For any $r>0$, we define $M_r$ as the interior component of  $M\setminus \Phi ^{-1}( S_r) $, where $S_r$ denotes the standard sphere with radius $r$ in $\mathbb{R}^3$.

\begin{thm}\label{main-scalar}
	Given some constants $A, B >0$, and $\sigma >\frac{1}{2}$, consider a sequence of orientable complete $(A, B, \sigma )$-AF $3$-manifolds $(M_i^3, g_i)$  with nonnegative scalar curvature. If the ADM masses $m(g_i)$ of $(M_i, g_i)$ converge to zero, then there exist uniform constants $L_0, C>0$, depending only on $(A, B, \sigma )$, such that for any sufficiently small positive number $ \varepsilon$, up to a subsequence, there exist open subsets $Z_i \subset M_{i, L_0}$ with boundary area satisfying $\mathrm{Area}_{g_i}(\partial Z_i) \leq C m(g_i)^{\frac{1}{2}- \varepsilon }$, and for any base point $p_i \in M_i\setminus Z_i$, $$(M_i\setminus Z_i, g_i, p_i)\to (\mathbb{R}^3, g_E, 0)$$ in the  pointed $C^0$ modulo negligible volume sense as defined in Definition \ref{C^0 modulo}.
\end{thm}

\begin{rmk}
	Recently, in a paper co-authored with Antoine Song \cite{DongSong23}, we have proven a stronger version of this theorem. Namely, we have removed the uniformly asymptotically flatness assumption, improved the bound of the boundary area to be almost optimal, and established the pointed measured Gromov-Hausdorff convergence for reduced length metrics. As a result, we have confirmed Huisken-Ilmanen's original conjecture.
\end{rmk}

If we assume stronger conditions on the Ricci curvature, we can improve the $C^0$-convergence modulo negligible volume to the Gromov-Hausdorff convergence. The following theorem is the same result as in \cite{KKL21} but without their topological assumption.

\begin{thm}\label{main-ricci}
	Under the same conditions as in Theorem \ref{main-scalar}, for a fixed $\Lambda >0$, if the Ricci curvature are bounded uniformly from below by $\Ric_{g_i} \geq -2\Lambda$, then up to a subsequence, for any base point $p_i \in M_i \setminus M_{i, L_0}$, $(M_i, g_i, p_i)$ converges to $(\mathbb{R}^3, g_E, 0)$ in the pointed Gromov-Hausdorff topology.
\end{thm}

Now, let's discuss the ideas behind the proof of the stability of the positive mass theorem and provide an outline of this paper. First, we recall the result obtained by Kazaras-Khuri-Lee in \cite{KKL21}. In their work, they established Gromov-Hausdorff convergence under additional conditions, namely $H_2(M^3,\mathbb{Z})=0$ and Ricci curvature bounded uniformly from below. The main ingredient utilized by Kazaras-Khuri-Lee is the mass inequality derived by Bray-Kazaras-Khuri-Stern in \cite{BKKS22}. Recall that for any AF $3$-manifod $(M, g)$, the exterior region $M_{ext}$ is an AF $3$-manifold, possibly with boundary, and it satisfies the condition of containing no other compact minimal surfaces except for its boundary (for more details, see Section \ref{AE-harmonic-coord}).
The mass inequality, which will be further discussed in Section  \ref{AE-harmonic-coord}, is as follows:
$$
m(g) \geq \frac{1}{16\pi}\int_{M_{ext}}\left( \frac{|\nabla ^2u^j|^2}{|\nabla u^j|}+R_g|\nabla u^j| \right) d \mathrm{vol}_g,
$$ 
where $\{u^j\}_{j=1}^{3}$ are harmonic functions defined on $M_{ext}$ and asymptotic to three asymptotically flat coordinate functions, and the integral is taken over regular points of each $u^j$. It should be noted that
when $H_2(M,\mathbb{Z})=0$, this mass inequality holds on the entire manifold $M$. 

Under the assumption of a Ricci lower bound, Kazaras-Khuri-Lee employed Cheng-Yau's gradient estimate and various techniques from Cheeger-Colding's theory, such as the segment inequality, to derive an almost pointwise estimate for the functions $\{u^j\}_{j=1}^3$ from the mass inequality. These functions $\{u^j\}_{j=1}^3$ were then chosen as splitting functions to prove that the manifold converges to $\mathbb{R}^3$ as the mass tends to zero. 

In this paper, we aim to study the stability of the positive mass theorem without imposing these additional conditions. Indeed, there are two main difficulties when attempting to generalize the arguments of Kazaras-Khuri-Lee. Firstly, if we drop the condition $H_2(M,\mathbb{Z})=0$, the harmonic functions $\{u^j\}_{j=1}^3$ are only defined on the exterior region $M_{ext}$ and satisfy the Neumann boundary conditions. As a result, Cheng-Yau's gradient estimate does not work and we have no control over the behavior of $|\nabla u^j|$ near the boundary $\partial M_{ext}$. This lack of control prevents us from directly using $\{u^j\}_{j=1}^3 $ as splitting functions in the proof. Secondly, if we remove the lower bound on the Ricci curvature, it becomes challenging to control the behavior of the metric. Without such curvature bound, the metric in $M_{ext}$ could either concentrate or degenerate, making it difficult to establish the desired Gromov-Hausdorff convergence. Given these challenges, a different approach is required to address the general case of the stability problem. 

To overcome these difficulties, we will remove an open subset and focus on the convergence within the remaining closed subset, which we refer to as the ``regular subregion''. More precisely, the regular subregion $E^{\tau }$ is defined as a connected component of
$$
\{x \in M_{ext}: \sum_{j,k=1}^3 |\left<\nabla u^j, \nabla u^k \right>-\delta _{jk}|^2(x) \leq \tau \},
$$ 
which contains the end. In Proposition \ref{regular-subregion}, by utilizing the co-area formula and the mass inequality, for a suitable $\tau $, we show that this regular subregion is well-defined and possesses a smooth boundary with a small area. To ensure $\partial E^{\tau }$ has uniformly small area, we establish uniform estimates for $\{u^j\}$ in Section \ref{effective-estimate}.

When restricted to the regular subregion, the harmonic map $$\mathcal{U}=(u^1, u^2, u^3): E^{\tau }\to Y^{\tau }\subset \mathbb{R}^3$$ is a diffeomorphism onto its image. Moreover, under this map, the metric $g$ is close to the Euclidean metric in the $C^0$-sense. By applying the integral by parts to the integral of $|\nabla u^j|^2$, and considering the small area of $\partial E^{\tau }$, we can demonstrate that $E^{\tau }$ has an almost Euclidean volume. Finally, in Section \ref{flat-convergence}, we prove our main result that, up to diffeomorphism $\mathcal{U}$, the subregion $E^{\tau }$ converges to $\mathbb{R}^3$ in the $C^0$ modulo negligible volume sense.

In the case when the Ricci curvature has a uniform lower bound, we can always find a subsequence that converges to a limit space in the pointed Gromov-Hausdorff topology, denoted by $(M_i, g_i, p_i) \to (X,d_X,p_X)$. Due to the small area of $\partial E^{\tau_i }$, we can use the Bishop-Gromov volume comparison theorem to perturb the geodesic segments between points in $E^{\tau_i }$ and avoid $\partial E^{\tau_i }$. This perturbation allows us to show that $(E^{\tau_i }, d_i, p_i)$ equipped with the restricted metric converges to $(X,d_X,p_X)$. Furthermore, based on the effective estimates on the outer region, we know that the outer region in $X$ is isometric to the outer region in $\mathbb{R}^3$. We then fill the interior region of $X$ by using geodesic segments between points in the outer region. To ensure that all points in the interior region of $X$ can be obtained in this way, we rely on the structural theory of three-dimensional Ricci limit spaces, particularly the facts that $X$ is a topological manifold and geodesics are non-branching. This filling process gives us an isometric embedding 
$$
\Xi : X \to \mathbb{R}^3.
$$ 

Intuitively, one might think of this map $\Xi $ as the limit of the harmonic maps $\mathcal{U}_i: E^{\tau _i}\to \mathbb{R}^3$. However, since $E^{\tau _i}$ may not be a length space and it is not clear whether $\mathcal{U}_i$ is Lipschitz on $E^{\tau _i}$, we can not directly take the limit of $\mathcal{U}_i$. Therefore, we need the aforementioned perturbation and filling process to overcome this difficulty.
All these arguments are included in Section \ref{removable-sing}.

\subsection*{Notations}
We use $C$ to denote a uniform constant, which may vary from line to line. $\Psi (\varepsilon ), \Psi (\varepsilon | a)$ represent small uniform numbers that satisfy $\lim_{\varepsilon \to 0}\Psi (\varepsilon )=0,\ \lim_{\varepsilon \to 0}\Psi (\varepsilon |a)=0$ for each fixed $a$, repectively. $g_E$ denotes the Euclidean metric and $d_E$ denotes the induced distance. $[xy]$ represents the geodesic segment between $x,y \in \mathbb{R}^3$ with respect to metric $g_E$, and $|xy|$ represents the distance. $B(p,r)$ denotes the geodesic ball around point $p \in (M,g)$, and $B(0,r)$ denotes the Euclidean ball in $\mathbb{R}^3$ with center $0$. $|\Sigma|$ represents the area of the surface $\Sigma \subset \mathbb{R}^3$ under the metric $g_E$, and  $|\Omega |$ represents the volume of the domain $\Omega \subset \mathbb{R}^3$ under the metric $g_E$.
\subsection*{Acknowledgements}
The author would like to express gratitude to his advisor, Prof. Xiuxiong Chen, for his encouragement and support. The author thanks Prof. Hubert Bray, Prof. Marcus Khuri and Prof. Antoine Song for their interests in this work and valuable discussions, and Hanbing Fang for introducing the reference \cite{Hein10}. The author also thanks the referee for providing very helpful and detailed comments and suggestions that led to the new Definition \ref{C^0 modulo}.

\section{Preliminaries}
\subsection{Asymptotically flat $3$-manifolds}\label{AE-harmonic-coord}
In this section, we recall some basic facts and estimates of an asymptotically flat (AF) $3$-manifold. Let $(M^3, g)$ be a complete AF  $3$-manifold. According to Lemma 4.1 in \cite{HuiskenIlmanen01}, there exists a trapped compact region $T$ inside $M$, whose topological boundary consists of smooth embedded minimal $2$-spheres. The exterior region, denoted as $M_{ext}$, is defined as the metric completion of a connected component of $M \setminus T$ associated with one end. It is worth noting that  $M_{ext}$ is connected and AF, possesses a compact minimal boundary, and does not contain any other compact minimal surfaces (even immersed).

We have the following Riemannian Penrose inequality, which was proved by Huisken-Ilmanen \cite{HuiskenIlmanen01} for each connected component of $\partial M_{ext}$ and by Bray \cite{Bray01} for the total area.

\begin{thm}\label{Penrose-ineq}
	Let $(M^3, g)$ be a complete AF $3$-manifold with nonnegative scalar curvature and total mass $m(g)$. For an exterior region $M_{ext}$ associated with one end, the total area of its boundary is bounded from above by 
	$$
	\mathrm{Area}(\partial M_{ext}) \leq 16 \pi m(g)^2.
	$$ 
\end{thm}

Given an AF exterior region $(M_{ext}, g)$ associated with one end, let $\{x^j\}_{j=1}^3$ denote the given asymptotically flat coordinate functions in the end. Then there exist functions $u^j \in C^\infty(M_{ext})$, $j=1,2,3$, satisfying 
\begin{align}\label{harmonic}
\Delta _g u^j =0 \text{ in } M_{ext},\ \frac{\partial u^j}{\partial \nu }=0 \text{ on } \partial M_{ext},\ |u^j - x^j| = o(|x|^{1- \sigma }) \text{ as } |x|\to \infty,
\end{align}
where the Neumann boundary condition can be omitted if $\partial M_{ext}= \emptyset$. We refer to $u^j$ as a harmonic function asymptotic to an AF coordinate function $x^j$. Further details can be found in \cite{BKKS22, KKL21}.

The following mass inequality was proven in \cite{BKKS22}, whose proof was based on the techniques introduced by Stern in \cite{Stern22}.

\begin{thm}[Theorem 1.2 in \cite{BKKS22}]\label{integral-ineq}
	Let $(M_{ext}, g)$ be an AF exterior region with mass $m(g)$. Let $u$, satisfying (\ref{harmonic}), be a harmonic function on $(M_{ext}, g)$ asymptotic to one of the AF-coordinate functions of the end. Then 
	\begin{align}\label{mass-ineq}
	m(g) \geq \frac{1}{16\pi}\int_{{M_{ext}}} \left( \frac{|\nabla ^2 u|^2}{| \nabla u|} + R_g |\nabla u| \right) \mathrm{dvol}_g,
\end{align} 
	where the integral is taken over regular points of $u$.
\end{thm}

Also, by Proposition 3.1 and Lemma 3.2 in \cite{KKL21}, we have the following estimates for harmonic functions satisfying (\ref{harmonic}).

\begin{prop}\label{C^2-estimate}
	Let $(M, g)$ be an $(A, B, \sigma )$-AF $3$-manifold with $R_g \geq 0$ and mass $m(g) < m_0$. Assume $u$ is a harmonic function on $M_{ext}$ defined by (\ref{harmonic}), asymptotic to an AF coordinate function $x^j$ and satisfies the Neumann boundary condition on $\partial M_{ext}$. Then there exists a uniform big $r_0=r_0(A, B, \sigma )>0$ such that for any $r >r_0$ and small $\delta >0$, after the normalization such that the average of $u$ over $M_r \setminus M_{r_0}$ is zero, we have
	$$
	\sup_{M_{r}}|u| \leq C(r),$$
	where $C(r)$ is a uniform constant depending only on $(A,B,\sigma , m_0, r)$. 
	Moreover, on $M \setminus M_{r_0}$, we have
	$$
	|\nabla u - \partial _{x^j}|(x) \leq C(r_0)|x|^{-\sigma }.
	$$
\end{prop}

\subsection{Space with Ricci curvature bounded from below}
Given a metric space $(Z, d)$ and two subsets $A, B \subset Z$, the Hausdorff distance $d_H(A, B)$ between $A$ and $B$ is defined as 
$$
d_H(A, B):= \inf \{\varepsilon : A \subset B_\varepsilon (B), B \subset B_\varepsilon (A)\} ,
$$ 
where $B_\varepsilon (A ):= \{ x \in Z: d(x, A) < \varepsilon \} $. 
For two metric spaces $X$ and $ Y$, an admissible metric on the disjoint union $X\sqcup Y $ is a metric that extends the given metrics on $X$ and $Y$. With this the pointed Gromov-Hausdorff distance is defined as $$
d_{GH}( (X,x), (Y,y) ) = \inf\{d_H(X,Y)+|xy|: \text{ admissible metrics on } X \sqcup Y\}.
$$ 

Alternatively, we can use the Gromov-Hausdorff approximation (GHA for short) to define the Gromov-Hausdorff distance. A map $f:X\to Y$ is called an $\varepsilon $-GHA if it satisfies the following conditions:
\begin{itemize}
	\item [(1)] $Y \subset B_\varepsilon (f(X) )$;
	\item[(2)] $|d(x_1, x_2) - d(f(x_1), f(x_2) )|<\varepsilon $, $\forall x_1, x_2 \in X$.
\end{itemize}
Define
$$
\hat{d}_{GH}( (X, x), (Y,y) ) = \inf \{\varepsilon : \exists \varepsilon\text{-GHA } f:X\to Y, f(x)=y\} .
$$ 
Then $\frac{1}{2}d_{GH} \leq \hat{d}_{GH} \leq 2 d_{GH}$ (see for example \cite{Rong06} for a proof).

For a compact metric space $X$, define the capacity and covering functions by 
$$
\mathrm{Cap}_X(\varepsilon ) = \text{ maximum number of disjoint }\frac{\varepsilon }{2}\text{-balls in }X,
$$ 
$$
\mathrm{Cov}_X(\varepsilon ) = \text{ minimum number of }\varepsilon \text{-balls it takes to cover }X.
$$ 
\begin{prop}[Gromov's compactness]
	For a class $\mathcal{C}$ of compact metric spaces with uniformly bounded diameter, the following statements are equivalent:

	1) $\mathcal{C}$ is precompact for the Gromov-Hausdorff topology;

	2) there is a function $N_1(\varepsilon ): (0, \alpha )\to (0, \infty)$ such that  $\mathrm{Cap}_X(\varepsilon ) \leq N_1(\varepsilon )$ for all $X \in \mathcal{C}$;

	3) there is a function $N_2(\varepsilon ): (0, \alpha )\to (0, \infty)$ such that $\mathrm{Cov}_X(\varepsilon ) \leq N_2(\varepsilon )$ for all $X \in \mathcal{C}$.
\end{prop}

See \cite[Section 11.1]{Petersen16} for a proof and more details of the above proposition. As corollaries, we have the following lemmas.

\begin{lem}\label{cpt-subset-converg}
	If $(X_i,d_i)$ are metric spaces with diameter bounded uniformly from above and $(X_i, d_i) \to (X, d)$ in the Gromov-Hausdorff topology, then for any compact subsets $\Omega _i \subset X_i$ with restricted metrics, up to a subsequence, $(\Omega _i, d_i) \to ( \Omega ,d) \subset (X, d)$ for some compact subset $\Omega \subset X$ in the Gromov-Hausdorff topology.
\end{lem}

\begin{lem}\label{precomp-ricci}
	The collection of complete pointed Riemannian $n$-manifolds with Ricci curvature bounded uniformly from below is precompact in the pointed Gromov-Hausdorff topology.
\end{lem}

For Riemannian manifolds with Ricci curvature bounded from below, we have the Bishop-Gromov's volume comparison theorem (see \cite[Section 7.1]{Petersen16} for a proof).
\begin{thm}
	Let $(M,g,p)$ be a complete Riemannian $n$-manifold with $\Ric_g \geq -(n-1)\Lambda g$ for some $\Lambda \geq 0$. Then for any $0<r_1<r_2$, we have
	$$
	\frac{\mathrm{Vol}_g B(p, r_1)}{|B_{-\Lambda }(r_1)|}\geq \frac{\mathrm{Vol}_g B(p, r_2)}{|B_{-\Lambda }(r_2)|},
	$$ where $|B_{-\Lambda }(r)|$ is the volume of a geodesic ball with radius $r$ in the space form $S_{-\Lambda }^n$ with constant curvature $-\Lambda $. Moreover, for $v \in T_p M$ a unit tangent vector and $\gamma _v(t), t \in [0,l]$ a minimal geodesic segment with $\gamma _v(0)=p, \dot\gamma (0) = v$, and for any $0 < t_1< t_2<l$, we have
\begin{align}\label{bishop-gromov}
	\frac{\mathcal{A}(t_1v)}{\underline{\mathcal{A}}(t_1)} \geq \frac{\mathcal{A}(t_2v)}{\underline{\mathcal{A}}(t_2)},
\end{align} 
	where $\mathcal{A}(tv)$ is the area element of $\partial B(p, t)$ in geodesic polar coordinates at $\gamma _v(t)$ and $\underline{\mathcal{A}}(t)$ is the area element of $\partial B(t)$ in $S_{-\Lambda }^n$. 
\end{thm}

By the Bishop-Gromov's volume comparison theorem, we have the following lemma. Our argument follows from Lemma 3.7 in \cite{Hein10}. We give full details here for the reader's convenience. 

\begin{lem}\label{vol-area}
	Assume $(M^n, g)$ is a complete Riemannian manifold with Ricci curvature bounded from below by $\Ric_g \geq -(n-1)\Lambda $, and $\Sigma^{n-1}\subset M$ is a smooth hypersurface. Given constants $D>1, \delta _1>0$, there exists a uniform constant $C= C(D, \delta _1, \Lambda , n)>0$ such that the following holds: for any point $p \in M$ and any fixed point $ x \in B(p, D)$ with $d(x, \Sigma) > 2\delta _1$, and for any subset $A \subset B(p, D)$, if for all points $z \in A$, there exists a geodesic segment $\gamma _{xz}$ between $x$ and $z$ such that $\gamma _{xz} \cap \Sigma \neq \emptyset$, then 
	$$
	\mathrm{Vol}(A) \leq C \cdot \mathrm{Area}(\Sigma \cap B(p, 2D) ).
	$$ 

\end{lem}
\begin{proof}
By assumptions, for any point $z \in A$, there exists a geodesic segment $\gamma _{xz}$ such that $\gamma _{xz}(0)= x$ and $\gamma _{xz} \cap \Sigma \neq \emptyset$. 
	Let $\Sigma^1 \subset \Sigma$ be the subset of all points $y \in \Sigma$ which occur as the first intersection with $\Sigma$ of $\gamma _{xz}$ for all $z \in A$. Define $d_1, d_2: \Sigma^1\to \mathbb{R}^+$ by $d_1(y)=d(x,y)$ and $d_2(y)= \sup \{t>0: \gamma _{xy}(d_1(y)+t) \in A\} $, where we use $\gamma _{xy}$ to denote the geodesic $\gamma _{xz}$ containing points $x$ and $y$.  Also define $$U:= \{(y, t) \in \Sigma^1 \times \mathbb{R}: d_1(y)<t \leq d_1(y)+d_2(y)\} ,$$ and $\Phi : U\to M$ by $\Phi (y,t)= \gamma _{xy}(t)$. 

	Then $\delta _1\leq  d_1(y) \leq d_1(y)+ d_2(y) \leq 2D$ and
\begin{align}\label{Phi^*}
	\Phi ^*\left( d\Vol _g \right) |_{(y,t)}= \frac{\mathcal{A}(tv_y)}{\mathcal{A}(d_1(y) v_y)}\cos \alpha _y dA_{\Sigma}\wedge dt,
\end{align}
where $v_y=\dot\gamma _{xy}(0)$ and $\alpha _y$ is the angle between $\dot\gamma _{xy}(d_1(y) )$ and the exterior normal to $\Sigma$ at $y$. Integrating (\ref{Phi^*}) over $U$ gives
	\begin{align*}
		\Vol(A) & \leq \int_{\Sigma^1}\int_{d_1(y)}^{d_1(y)+d_2(y)} \frac{\mathcal{A}(tv_y)}{\mathcal{A}(d_1(y)v_y)}dt dA(y)\\
			 & \leq \int_{\Sigma^1}\int_{d_1(y)}^{d_1(y)+d_2(y)} \frac{\underline{\mathcal{A}}(t)}{\underline{\mathcal{A}}(d_1(y) )}dt dA(y)\\
			 & \leq C(D, \delta _1, \Lambda , n) \mathrm{Area}(\Sigma \cap B(p, 2D) ).
	\end{align*}
where we used the Bishop-Gromov's volume comparison inequality (\ref{bishop-gromov}) in the above second inequality.
\end{proof}

As a corollary, we have the following perturbation lemma. 
\begin{lem}\label{nbhd-point}
	Under the same assumptions as in the above lemma, and assume further that  $\mathrm{Vol}_g( B(p, 1) ) \geq v_0> 0$. 
	If $\mathrm{Area}(\Sigma \cap B(p, 2D) )\leq \varepsilon $,
	then there exists $\delta _0=\Psi (\varepsilon | n, v_0, \Lambda , D, \delta _1)>0 $ such that the following holds: for any points $x, w \in B(p, D)$ with $d(x, \Sigma) > 2\delta _1$, there exists a point $z \in B(w, \delta_0 )$ such that for any geodesic segment $\gamma _{x z}$ connecting $x$ and $z$, $\gamma _{x z} \cap \Sigma = \emptyset.$

\end{lem}
\begin{proof}
	Otherwise, for any point $z \in B(w, \delta _0)$, there exists a geodesic segment $\gamma _{x z}$ such that $\gamma _{x z} \cap \Sigma \neq \emptyset$. By Lemma \ref{vol-area}, $$\mathrm{Vol}( B(w, \delta _0) ) \leq C \cdot \mathrm{Area}(\Sigma \cap B(p, 2D) )$$ for some uniform constant $C>0$. However, by the Bishop-Gromov's volume comparison theorem, 
	$$
	\Vol( B(w, \delta _0) ) \geq \frac{| B_{-\Lambda }(\delta _0) |}{|B_{-\Lambda }(3D)|}\cdot \Vol( B(p, D) ) \geq \frac{| B_{-\Lambda }(\delta _0)|}{| B_{-\Lambda }(3D)|}\cdot v_0.
	$$
	Choosing $\delta _0$ such that $\frac{| B_{-\Lambda }(\delta _0)|}{| B_{-\Lambda }(3D)|}\cdot v_0 = 2C \varepsilon $ gives a contradiction.
\end{proof}
\begin{rmk}
	Under the same notations as in the above lemma, if we denote the subset in $B(w, \delta _0)$ consisting of points which could be connected to $x$ by segments in $M\setminus \Sigma$ as $\hat{A}$, then by Lemma \ref{vol-area}, we know that $\Vol(\hat{A}) \geq (1- \sqrt{\varepsilon } )\Vol( B(w, \delta _0) )$ when $\varepsilon \ll 1$. Therefore, for any uniform positive number $N$ and points $x_1, \cdots, x_N$ with $d(x_j, \Sigma) > 2 \delta _1$ for any $j$, if $\varepsilon $ is small enough depending on $N$, by the same arguments, there exists $z \in B(w, \delta _0)$ such that each geodesic segment connecting $z$ and $ x_j$ lies in $M\setminus \Sigma$.
\end{rmk}

We also require the structure theory of the Ricci limit space. The following proposition corresponds to a specific case of Theorem 1.3 in \cite{Deng20}.

\begin{prop}\label{non-branching}
	If $(M_i, g_i, p_i)$ is a sequence of Riemannian manifolds with Ricci curvature bounded uniformly from below, and $(M_i, g_i, p_i)$ converge to $(X, d, p)$ in the pointed Gromov-Hausdorff topology, then $(X,d)$ is non-branching. In other words, if there are two geodesics $\gamma _1$ and $\gamma _2$ in $X$, and for some $t>0$, $\gamma _1(s)=\gamma _2(s)$ for $ s \in [0,t]$, then $\gamma _1=\gamma _2$.
\end{prop}

In dimension three,  the following theorem \cite[Corollary 1.5]{SimonTopping21} implies that any Ricci limit space is a manifold.

\begin{thm}\label{mfd-structure}
	If $(M_i^3, g_i, p_i)$ is a sequence of Riemannian $3$-manifolds with Ricci curvature bounded uniformly from below and satisfy the noncollapsing conditions that  $\Vol_{g_i} B(p_i, 1) \geq v_0$ for a uniform $v_0 >0$, and $(M_i, g_i, p_i)$ converge to $(X,d,p)$ in the pointed Gromov-Hausdorff topology, then there exists a topological $3$-manifold $M$ such that $M=X$ and the topology generated by $d$ is the same as $M$. Moreover, the charts of $M$ can be taken to be bi-H\"older with respect to $d$.
\end{thm}

\section{Effective estimates on outer region}\label{effective-estimate}
In this section, we always assume that $(M, g)$ is an exterior region of an orientable complete pointed AF $3$-manifold with nonnegative scalar curvature $R_g \geq 0$, and the mass $m(g)<1$. 
We define $\{u^j\}, j=1,2,3$ as harmonic functions asymptotic to the AF-coordinate functions $x^j$. Recall that these harmonic functions are defined over $M$ and satisfy the Neumann boundary condition on $\partial M$. For the rigorous definitions and a more detailed explanation, please refer to Section \ref{AE-harmonic-coord}.

Also we assume that $(M,g)$ is $(A,B, \sigma )$-AF and all constants used are assumed to depend on $(A,B, \sigma )$ but not on any thing else (unless explicitly stated otherwise).

Recall that $M_r$ is defined as the interior component of $M \setminus \Phi ^{-1}(S_r)$, where $\Phi $ is the $C^\infty$-diffeomorphism from the end into $\mathbb{R}^3$ and $S_r \subset \mathbb{R}^3$ is the standard sphere with radius $r$.

\begin{lem}\label{hessian estimate}
	There exists a sufficiently large uniform $r_0>0$ and a uniform constant $C>0$ such that for any point $x \in M\setminus M_{r_0}$, 
	$$
	|\nabla ^2 u^j|(x) \leq C\cdot  m(g)^{\frac{1}{20}}= \Psi (m(g)) \text{ for }j\in \{1,2,3 \}.
	$$ 
\end{lem}
\begin{proof}
Fix a $j \in \{1,2,3\}$ and denote $u^j$ by $u$ for simplicity.
	For a sufficiently large uniform $r_0 >0$ and any point $x \in M\setminus M_{r_0}$, we know $$B(x,2) \subset \Phi ^{-1}(\mathbb{R}^3\setminus B(0,A) ),$$
	where the metric $g$ is $C^2$-close to the Euclidean metric, i.e. for any point $ y \in B(x,2)$, and for all multi-index $|l|=0,1,2$,
	$$|\partial^{l } (g_{ik}-\delta _{ik})|(y) \leq B|y|^{-\sigma- |l| }\leq 2B r_0^{-\sigma }\ll 1.$$
By Proposition \ref{C^2-estimate}, we have the gradient estimate that
$$
\sup_{B(x,1)}|\partial u|\leq C.
$$ 
By the Schauder estimate, $|\partial ^2 u| \leq C$ over $B(x,1)$, and thus 
\begin{align}\label{gradient+hessian}
	|\nabla u|(y) + |\nabla ^2 u|(y) \leq C, \ \forall y \in B(x, 1).
\end{align}
Notice that 
\begin{align}\label{bochner}
	\begin{split}
		\frac{1}{2}\Delta |\nabla ^2 u| ^2 &= |\nabla ^3 u|^2 + \nabla _j \nabla _i u \cdot  \nabla _k(R(\partial _{x^k}, \partial _{x^j})(\nabla _i u) )\\
						   &\ \ + \nabla _j \nabla _i u \cdot \nabla _j( R(\partial _{x^k}, \partial _{x^i})(\nabla _k u) )   + R_{lkij} \cdot \nabla _j \nabla _i u\cdot \nabla_k \nabla _l u.
	\end{split}
	\end{align} 
	Now we can apply the Moser iteration to control $|\nabla ^2 u|$ by its $L^2$-integral as follows. For any cut-off function $\varphi $ with support in $B(x,1)$, by using the identity (\ref{bochner}) and integration by parts, we have
	\begin{align}\label{bochner-integral}
		\begin{split}
			-\frac{1}{2}\int_{B(x, 1)}\nabla \varphi \cdot \nabla & |\nabla ^2 u|^2 \mathrm{dvol}_g \geq\\
	&\ \ - \int_{B(x,1)}\left( \varphi |\nabla u|^2 + |\nabla \varphi | |\nabla ^2 u| |\nabla u| + \varphi |\nabla ^2 u|^2 \right)\mathrm{dvol}_g ,
\end{split}
\end{align}
	where we used the fact that $|Rm| \ll 1$ over $B(x,1)$. 

	Take a cutoff function $\eta  $ with support in $B(x, 1)$ and a constant $\beta \geq 2$, which will be determined below. By choosing $\varphi =\eta ^2 |\nabla ^2 u| ^{2(\beta -1)}$ in (\ref{bochner-integral}), we have
	\begin{align*}
		(\beta& -1)\int_{}\eta ^2 |\nabla ^2 u|^{2(\beta -2)}|\nabla |\nabla ^2 u|^2|^2 + \int_{}2\eta |\nabla ^2u|^{2(\beta -1)}\nabla \eta \cdot \nabla |\nabla ^2 u|^2 \\
		      & \leq \int_{} \eta ^2 |\nabla ^2 u|^{2(\beta-1) }|\nabla u|^2 + \int_{} \eta |\nabla \eta | |\nabla ^2 u|^{2(\beta -1)}\cdot |\nabla ^2 u| |\nabla u|\\
		      &\ \ \ \  + \int_{}(\beta -1)\eta ^2 |\nabla ^2 u|^{2(\beta -2)} |\nabla |\nabla ^2 u|^2 |\cdot  |\nabla ^2 u| |\nabla u| + \int_{}\eta ^2 |\nabla ^2 u|^{2 \beta }.
	\end{align*}
	By the H\"older inequality and (\ref{gradient+hessian}), we have
	\begin{align*}
		\begin{split}
		(\beta -1)\int_{}\eta ^2 |\nabla ^2u|^{2(\beta -2)}| \nabla |\nabla ^2 u|^{2}|^2& \leq \frac{C}{\beta -1}\int_{}|\nabla \eta |^2 |\nabla ^2 u|^{2 \beta }+  \int_{}\eta ^2 |\nabla ^2 u|^{2 \beta }\\
												&\ \ + C \int_{}(\beta \eta ^2 |\nabla ^2 u|^{2(\beta -1)}+\eta |\nabla \eta | | \nabla ^2 u|^{2\beta -1}).
	\end{split}
	\end{align*}
	Set $v=|\nabla ^2u|^2$, then we have $v \leq C$ by (\ref{gradient+hessian}). So 
\begin{align*}
	\int_{}| \nabla (\eta v^{\frac{\beta}{2}})|^2& \leq C \int_{}|\nabla \eta |^2 v^{\beta }+ C \beta \int_{}(\eta v^{\frac{\beta }{2}})^2 + C \beta ^2 \int_{}\eta ^2v^{\beta -1}+ C \beta \int_{} \eta |\nabla \eta | v^{\beta -\frac{1}{2}}\\
						     &\leq C \beta ^2 \int_{}(\eta ^2+ \eta |\nabla \eta | + |\nabla \eta |^2) v^{\beta -1}.
\end{align*}
By the Sobolev inequality, 
\begin{align}\label{sobolev}
	\| \eta v^{\frac{\beta }{2}}\|_{L^6}^2 \leq C\cdot \| \eta v^{\frac{\beta }{2}}\|_{W^{1,2}}^2 \leq C\beta ^2 \int_{}( \eta ^2+ \eta |\nabla \eta |+ |\nabla \eta |^2 ) v^{\beta -1}.
\end{align}

Let $\eta \geq 0$ be a cutoff function satisfying $\eta (s) =1$ for $ s \leq 0$, $\eta (s) =0$ for $ s \geq 1$, and $|\eta '| \leq 2$. For any positive integer $k \geq 1$, set $r_k = \frac{1}{2}+ \frac{1}{2^{k+1}}$, $\Omega _k= B(x,r_k)$, and $$
\eta _k(y) = \eta \left( \frac{d(x,y) - r_k}{r_{k-1}-r_k} \right) .
$$ 
Then (\ref{sobolev}) implies that
\begin{align}
	\begin{split}
	\|v\|_{L^{3 \beta }(\Omega _k)}& \leq \|\eta _k v ^{\frac{\beta }{2}}\|_{L^6(\Omega _{k-1})}^{\frac{2}{\beta }}\\
				       &\leq ( C \beta ^2 4^{k})^{\frac{1}{\beta }}\left(\int_{\Omega _{k-1}}v^{\beta -1}\right) ^{\frac{1}{\beta }}\\
				       & \leq (C|B(x,1)|^{\frac{1}{\beta }}\beta ^2 4^{k})^{\frac{1}{\beta }}\| v\|_{L^\beta (\Omega _{k-1})}^{1-\frac{1}{\beta }}\\
				       & \leq (C \beta ^2 4^k)^{\frac{1}{\beta }}\|v\|_{L^\beta(\Omega _{k-1}) }^{1-\frac{1}{\beta }}.
\end{split}
\end{align}
Taking $\beta =2\cdot 3^{k-1}$ gives
\begin{align}\label{moser}
\begin{split}
	\|v\|_{L^{2 \cdot 3^k}(\Omega _k)} &\leq C ^{\frac{k}{3^k}}\|v\|_{L^{2\cdot 3^{k-1}}(\Omega _{k-1})}^{1-\frac{1}{2\cdot 3^{k-1}}}\\
					   & \leq C ^{\frac{k}{3^k}+ \frac{k-1}{3^{k-1}}(1-\frac{1}{2\cdot 3^{k-1}})}\|v\|_{L^{2\cdot 3^{k-2}}(\Omega _{k-2})}^{(1-\frac{1}{2\cdot 3^{k-1}})(1-\frac{1}{2\cdot 3^{k-2}})}\\
					   & \leq C ^{\sum_{j=1}^{k}\frac{j}{3^j}}\|v\|_{L^2(B(x,1) )}^{\prod _{j=0}^{k-1}(1-\frac{1}{2\cdot 3^{j}})}.
\end{split}
\end{align}
Notice that $$\sum_{j=1}^{\infty} \frac{j}{3^j} < \sum_{j=1}^{\infty}\frac{2^j}{3^j}<1,$$ $$\prod _{j=0}^{\infty} \left( 1- \frac{1}{2\cdot 3^j} \right) > \frac{1}{2}\cdot \frac{5}{6}\cdot \prod _{j=2}^{\infty}\left( 1- \frac{1}{j^2} \right) > \frac{1}{2}\cdot \frac{5}{6}\cdot \frac{1}{2}> \frac{1}{5}.$$ 
So (\ref{moser}) implies that
$$
\| v \|_{L^{2\cdot 3^{k}}(B(x, \frac{1}{2}))} \leq C \| v \|_{L^2(B(x,1) )}^{\frac{1}{5}}.
$$ 

Since
\begin{align*}
	\|v\|_{L^2(B(x,1) )}^2 &=
	\int_{B(x,1)}|\nabla ^2u|^4 \\
			       & \leq C \int_{B(x,1)}|\nabla ^2 u| \\
				    &\leq C \left( \int_{B(x,1)}\frac{|\nabla ^2 u|^2}{|\nabla u|} \right) ^{\frac{1}{2}}\left( \int_{B(x,1)}|\nabla u| \right) ^{\frac{1}{2}}\\
				    &\leq C\cdot m(g)^{\frac{1}{2}},
\end{align*}
by taking $k\to \infty$, we have
$$
\|v\|_{L^\infty(B(x,\frac{1}{2}) )} \leq C m(g)^{\frac{1}{20}}= \Psi (m(g)).
$$ 
\end{proof}

\begin{rmk}
	Notice that there is a ``loss of power'' in the Moser iteration argument used here, which is slightly different from the standard one. The reason for this can be seen from the following simple example: if $v \geq 0$ satisfies $- \Delta v \leq v + 1$ on the ball $B_1 \subset \mathbb{R}^n$, and $\| v\|_{C^0(B_1)} \leq C_1$, then we have $\|v\|_{C^0(B_{\frac{1}{2}})} \leq C(C_1) \|v\|_{L^2(B_1)}^{\frac{1}{5}}$. The power $\frac{1}{5}$ is not optimal. 
\end{rmk}

\begin{lem}\label{onb-est}
	There exists a sufficiently large uniform $r_0>0$ and a uniform constant $C>0$ such that for any point $x \in M \setminus M_{r_0}$, 
$$
|\left< \nabla u^j, \nabla u^k \right>-\delta _{jk}|(x) \leq Cm(g)^{\frac{1}{80}}= \Psi (m(g)) \text{ for } j, k \in \{1,2,3\} .
$$ 
\end{lem}
\begin{proof}
	Under the same notations as in the proof of Lemma \ref{hessian estimate}, by Proposition \ref{C^2-estimate},  we know that $ \forall x \in M\setminus M_{r_0}$, 
	$$
	|\nabla u^j - \partial x^j|(x) \leq C|x|^{-\sigma }.
	$$
	So
	\begin{align}\label{inner-prod}
		\begin{split}
			|\left< \nabla u^j \right. &,\left. \nabla u^k \right>- \delta _{jk}|\\
		&\leq |\left< \nabla u^j- \partial x^j, \nabla u^k \right> + \left<\partial x^j, \nabla u^k - \partial x^k \right>+ \left<\partial x^j, \partial x^k \right>- \delta _{jk}|\\
								      &\leq C |x|^{-\sigma }.
	\end{split}
        \end{align} 
	
	Choose $L = L(m(g) )>0$ to be sufficiently large, which will be determined below based on $m(g)$. 

	If $x \in M\setminus M_L$, then by (\ref{inner-prod}), we have
	\begin{align}\label{M_L}
	|\left<\nabla u^j, \nabla u^k \right>-\delta _{jk}|(x) \leq C L^{-\sigma }.
\end{align}

If $x \in M_L\setminus M_{r_0}$, then we can choose $y \in \partial M_{L}$ such that $|xy|\leq  L-r_0$ and the line segment $[xy] \subset M_L\setminus M_{r_0}$. So by Lemma \ref{hessian estimate} and (\ref{M_L}), we know
\begin{align}\label{M_r0}
	\begin{split}
	|\left<\nabla u^j, \nabla u^k \right>-\delta _{jk}|^2 (x) & \leq  |(\left<\nabla u^j, \nabla u^k \right>-\delta _{jk})^2 (x) -|\left<\nabla u^j, \nabla u^k \right>-\delta _{jk}|^2 (y)|\\
								  &\ \ + (\left<\nabla u^j, \nabla u^k \right>-\delta _{jk})^2 (y)\\
								  &\leq L \cdot C m(g)^{\frac{1}{20}} + CL^{-2\sigma }.
\end{split}
\end{align}

 Recall that $\sigma > \frac{1}{2}$. Taking $L= m(g)^{-\frac{1}{40}}$ in (\ref{M_L}) and (\ref{M_r0}), we know that for any point $ x \in M\setminus M_{r_0}$,
$$
| \left<\nabla u^j, \nabla u^k \right> - \delta _{jk} |(x) \leq C m(g)^{\frac{1}{80}} = \Psi (m(g) ).
$$ 
\end{proof}

For $C(r_0)$ in Proposition \ref{C^2-estimate}, we can choose $r_1$ such that $C(r_0) r_1^{-\sigma }< \frac{1}{2}$. We fix a normalization as in Proposition \ref{C^2-estimate} such that the average of $u^j$ over $M_{r_1}\setminus M_{r_0}$ is zero for each $j \in \{1,2,3\} $. Then $\sup_{M_{r_1}}|u^j| \leq C(r_1)$. Notice that together with the gradient estimate in Proposition \ref{C^2-estimate}, we also have that for all $r > r_0$, 
\begin{align}\label{C(r)}
\sup_{M_r}|u^j| \leq C(r).
\end{align}

Define $$\mathcal{U}:=(u^1, u^2, u^3): M \to \mathbb{R}^3.$$ 
If we choose $r_2= 2C(r_1)$, then 
\begin{align}\label{outer1}
	\mathcal{U}^{-1}(\mathbb{R}^3\setminus B(0, r_2) ) \subset M\setminus M_{r_1}.
\end{align}

At the end of this section, we will prove two lemmas that will be used later to show that $\mathcal{U}$ is one-to-one around the infinity in the end.

\begin{lem}\label{one-to-one-outer}
	There exists a uniform $L_0> r_2$ such that for all $L> L_0$, and for any $w \in \mathbb{R}^3 \setminus B(0, L)$, the preimage $\mathcal{U}^{-1}\{w\} $ consists of at most one point.
\end{lem}
\begin{proof}
	Assume that $\mathcal{U}(y)=\mathcal{U}(z)=w$ for some $w \in \mathbb{R}^3\setminus B(0, L)$. Then we know $y, z \in M\setminus M_{r_1}$ by (\ref{outer1}). Under the diffeomorphism $\Phi : M\setminus M_{r_1}\to \mathbb{R}^3$, we can assume that $y, z \in \mathbb{R}^3\setminus B(0, r_1)$, where the metric $g$ is $C^2$-close to the Euclidean metric. Now consider $\mathcal{U}$ as a mapping $\mathcal{U}: (\mathbb{R}^3\setminus B(0,r_1), g) \to \mathbb{R}^3$. Take the straight line segment $[yz]$ between $y$ and $z$, and assume that $y'$ and $ z'$ are points on $\partial B(0,r_1)$ that lie on $[yz]$ and are closest to $y$ and $z$, respectively. Denote the unit direction vector of $[yz]$ as $a= (a_1, a_2, a_3) \in \mathbb{S}^2$. Set $u^a = \sum_{j=1}^3a_j u^j$, and $l=|yz|$. Since $\mathcal{U}(y)= \mathcal{U}(z)$, $u^{a}(y) = u^{a}(z)$ for any $a$. Now we have two cases to consider.

	Case 1): If no such points $y', z'$ exist, then $[yz] \subset \mathbb{R}^3\setminus B(0,r_1)$. By Proposition \ref{C^2-estimate}, we have
	\begin{align}\label{case1}
		\begin{split}
		u^{a}(z) - u^{a}(y) &= \int_{0}^{l} (u^{a}(y+at) )'dt\\
				    &= \int_{0}^{l} \sum_{j,k=1}^3\left< a_j \nabla u^j, a_k \partial y^k\right> \\
				    & \geq (1- C(r_0) r_1^{-\sigma })l\\
				    & \geq  \frac{1}{2}l,
	\end{split}
	\end{align}
	 which implies that $l=0$, and particularly $y=z$.

	 Case 2): If such $y',z'$ exist, then by choosing $L_0$ sufficiently large, we can assume that $l> 2 C( r_1)$. In this case, we can construct a piecewise straight line segment $\gamma _{y'z'}$ between $y'$ and $z'$ such that $\gamma _{y'z'}\subset \mathbb{R}^3\setminus B(0, r_1)$ and $|\gamma _{y'z'}| \leq 10r_1$. Furthermore, consider the new piecewise straight line segment $\gamma := [yy'] \cup \gamma _{y'z'}\cup [z'z]$. It follows that $\gamma \subset \mathbb{R}^3\setminus B(0, r_1)$. Similar to (\ref{case1}), we have  
	$$
	u^{a}(z)- u^{a}(z') \geq \frac{1}{2}|zz'|,\ u^{a}(y')-u^{a}(y) \geq \frac{1}{2}|yy'|,
	$$ 
	and by using $|\gamma _{y'z'}| \leq 10 r_1$ and Lemma \ref{onb-est},
	$$
	|u^{a}(z') - u^{a}(y')| \leq C(r_1).
	$$ 
So
$$
u^{a}(z)- u^{a}(y) \geq \frac{1}{2}(|zz'|+|yy'|) - C(r_1) \geq \frac{1}{2}l - C(r_1)>0,
$$ 
which is a contradiction.
\end{proof}
\begin{lem}\label{nonempty-preimage}
	For any $L_0>0$, there exists uniform $L_1 = L_1(L_0)$ depending on $L_0$ such that for all $L > L_1$, 
	$$\mathcal{U}^{-1}( \mathbb{R}^3\setminus B(0, L) ) \subset M\setminus M_{L_0},\ \ \mathcal{U}(M\setminus M_{L}) \subset \mathbb{R}^3\setminus B(0, L_0).$$
\end{lem}
\begin{proof}
	By (\ref{C(r)}), we know $\sup_{M_{L_0}}|u^j| \leq C(L_0)$, which implies that for any $L > C(L_0)$, 
	$$\mathcal{U}^{-1}(\mathbb{R}^3\setminus B(0, L) ) \subset M\setminus M_{L_0}.$$

	For any $y \in M\setminus M_L$, we can choose a straight line segment $[xy]$ with $x \in \partial M_{L_0}$ being the closest point to $y$. Let $a$ denote the unit direction vector of $[xy]$. Similar to (\ref{case1}), we know that $$
	u^a(y)  \geq u^a(x) + \frac{1}{2}|xy| \geq  \frac{1}{2}(L-L_0) - C(L_0),
	$$ 
	which is greater than $L_0$ if $L > 3L_0 + C(L_0)$.

	Thus, by choosing $L_1= 3L_0+ C(L_0)$, the proof is complete.
\end{proof}

\section{$C^0$ modulo negligible volume convergence}\label{flat-convergence}

In this section, we will introduce the concept of a regular subregion, which is a region where the geometry is locally uniformly controlled. This regular subregion will play a crucial role in proving our convergence theorem. We will work within the same setup and utilize the notations introduced in Section \ref{effective-estimate}. Specifically, we assume that $M$ is an $(A, B, \sigma )$-AF exterior region and the mass satisfies $m(g)< 1$. 

By Sard's theorem, for any small $\tau \in (0,1)$ and $j \in \{1,2,3\} $, we can find $\tau^j_1 \in (\frac{\tau }{2}, \tau)$ such that the level set $$\{x \in M: |\nabla u^j|^2 (x) =1+ \tau^j_1 \}$$ is a smooth surface.

Define $$E_1^{\tau} := M\cap ( \cap _{j=1}^{3} \{x \in M: |\nabla u^j|^2(x)\leq 1+\tau^j_1  \} ).$$
Fix a uniform constant $r_0>0$ such that Lemma \ref{hessian estimate} and Lemma \ref{onb-est} hold. 
By the co-area formula, we have
\begin{align}\label{co-area}
	\begin{split}
	\int_{0}^{\infty} \mathcal{H}_{g}^2(&M_{r_0}\cap E^{\tau}_1 \cap \{ \sum_{j,k=1}^3| \left< \nabla u^j, \nabla u^k \right>-\delta _{jk}|^2 = t\} ) dt \\
					    &= \int_{M_{r_0}\cap E_1^{\tau }} \left|\nabla  \sum_{j,k=1}^3 | \left<\nabla u^j, \nabla u^k \right>-\delta _{jk}|^2 \right|\\
									 &\leq 2( (1+\tau )^2 +1) \int_{M_{r_0}\cap E_1^{\tau} }\sum_{j,k=1}^3(|\nabla ^2 u^j| |\nabla u^k|+ |\nabla ^2u^k| |\nabla u^j|) \\
									 & \leq C \sum_{j,k=1}^3 \left(\int_{M_{r_0}\cap E_1^{\tau }} \frac{|\nabla ^2 u^j|^2}{|\nabla u^j|}\right) ^{\frac{1}{2}} \left( \int_{M_{r_0}\cap E_1^{\tau }}|\nabla u^j|^{3} \right) ^{\frac{1}{6}}\left( \int_{M_{r_0}\cap E_1^{\tau }}|\nabla u^k|^{3} \right) ^{\frac{1}{3}} \\
									 &\leq C(r_0) m(g)^{\frac{1}{2}},
\end{split}
\end{align}
where in the last inequality, we used the mass inequality (\ref{mass-ineq}) and the fact that 
\begin{align*}
	\int_{M_{r_0}\cap E_1^{\tau }}|\nabla u^j|^{3} &\leq (1+\tau ) \int_{M_{r_0}}|\nabla u^j|^2\\
						       &\leq (1+ \tau )\int_{\partial B(0, r_0)} u^j |\nabla u^j|\\
						       &\leq C(r_0).
\end{align*}

For $\tau _1:= \min \{\tau _1^1, \tau _1^2, \tau _1^3\} $, by Sard's theorem again and using (\ref{co-area}), we can find a generic $\tau _2 \in (\frac{\tau_1 }{2}, \tau_1) $ such that the level set
$$
\{x \in M: \sum_{j,k=1}^3| \left< \nabla u^j, \nabla u^k \right>-\delta _{jk}|^2 = \tau_2\} 
$$ 
is a smooth surface and satisfies
\begin{align}\label{boundary-area-1}
\mathcal{H}^2_{g}(M_{r_0} \cap E_1^{\tau} \cap \{ x \in M: \sum_{j,k=1}^3| \left< \nabla u^j, \nabla u^k \right>-\delta _{jk}|^2 = \tau_2\}) \leq \frac{C(r_0)m(g)^{\frac{1}{2}}}{\tau}.
\end{align}

Define
\begin{align}\label{E^tau_2}
E_2^{\tau}:= M\cap \{x \in M: \sum_{j,k=1}^3| \left< \nabla u^j, \nabla u^k \right>-\delta _{jk}| ^2\leq \tau_2\} .
\end{align} 
Then $E_2^{\tau }\subset E_1^{\tau }$. Notice that $\partial M \cap E_2^{\tau }= \emptyset$. This is because on $\partial M$, $\left<\nabla u^j, \nu  \right>=0$ for the normal vector $\nu $ of $\partial M$ and all $j \in \{1,2,3\} $, but $\{\nabla u^j\} _{j=1}^3$ is an almost orthonormal basis at any point in $E^{\tau }_2$. So we have
$$
\partial E_2^{\tau }= \{x \in M: \sum_{j,k=1}^{3}|\left<\nabla u^j, \nabla u^k \right>-\delta _{jk}|^2 = \tau _2\} .
$$ 
\begin{lem}\label{E_2^tau}
	For any $ \varepsilon \in (0, \frac{1}{100})$, there exists a uniform constant $C>0$ such that for $\tau = m(g)^{\varepsilon }$ and for all $m(g) \ll 1$, 
	$$
	M\setminus M_{r_0}\subset E_2^{\tau },
	$$ 
	and
	$$\mathcal{H}^2_{g}(\partial E_2^{\tau })\leq C \cdot m(g)^{\frac{1}{2}-\varepsilon }= \Psi (m(g)).
$$
\end{lem}
\begin{proof}
	Since $C m(g)^{\frac{1}{80}} < \frac{\tau}{4} $ for all $m(g) \ll 1$, by Lemma \ref{onb-est}, we have
	$$M\setminus M_{r_0} \subset E_2^{\tau },\ \ \partial E_2^{\tau }\subset M_{r_0}.$$
	So (\ref{boundary-area-1}) implies that
\begin{align*}
	\mathcal{H}^2_g(\partial E_2^{\tau })=\mathcal{H}^2_g(M_{r_0}\cap \partial E_2^{\tau }) \leq C(r_0) m(g)^{\frac{1}{2}-\varepsilon }= \Psi (m(g)).
\end{align*} 
\end{proof}
For any fixed $\varepsilon \in (0, \frac{1}{100})$, let $E_2^{\tau }$ be the region defined in Lemma \ref{E_2^tau}. We then define the region $E^{\tau }_3$ as the connected component of  $E_2^{\tau }$ that contains $M\setminus M_{r_0}$. Therefore, $E^{\tau }_3$ also satisfies the properties in Lemma \ref{E_2^tau}. We consider the harmonic map $\mathcal{U}$ restricted to $E^{\tau }_3$, given by
$$
\mathcal{U}=(u^1, u^2, u^3): E^{\tau }_3\to \mathbb{R}^3.
$$ 

By definition, $\mathcal{U}$ is nondegenerate and therefore a local diffeomorphism over $E^{\tau }_3$. The regular subregion is defined as the subset of $E^{\tau }_3$ where $\mathcal{U}$ is a diffeomorphism as follows.

Set $Y^{\tau }_3:= \mathcal{U}(E^{\tau }_3)$, and $$Y^{\tau }_4 := \{ y \in Y^{\tau }_3: \mathcal{H}^0( \mathcal{U}^{-1}\{ y\} \cap E^{\tau }_3) =1 \},\ Y^{\tau }_s := Y^{\tau }_3\setminus Y^{\tau }_4. $$

\begin{lem}\label{Y^tau open}
	$Y^{\tau }_4$ is open in $Y^{\tau }_3$, and $\partial Y^{\tau }_4 \subset \mathcal{U}(\partial E^{\tau }_3)$.
\end{lem}
\begin{proof}
	Otherwise, if we assume that $Y^{\tau }_4$ is not open in $Y^{\tau }_3$, then for some $y \in Y^{\tau }_4\setminus \partial Y^{\tau }_3$, there exists a sequence of $y_k \to y$ and two different sequences of $x_k^1 \neq  x_k^2$ such that $\mathcal{U}(x_k^1)=\mathcal{U}(x_k^2)=y_k$. Assuming $\mathcal{U}(x)=y$, we have $x_k^1\to x$ and $x_k^2\to x$, which contradicts the fact that $\mathcal{U}$ is a local diffeomorphism. 

	To see $\partial Y^{\tau }_4 \subset \mathcal{U}(\partial E^{\tau }_3)$, assume otherwise that there exists $z \in \partial Y^{\tau }_4$ such that $z \notin \mathcal{U}(\partial E^{\tau }_3)$. Notice that $\partial Y^{\tau }_3 \subset \mathcal{U}(\partial E^{\tau }_3)$, so $z \notin \partial Y^{\tau }_3$. If $z \in Y_4^{\tau }$, then since $Y_4^{\tau }$ is open in $Y^{\tau }_3$ as argued before, $z \notin \partial Y_4^{\tau }$, which contradicts our assumption on $z$. Therefore, $z \in Y^{\tau }_s$, and there exist $x^1, x^2 \in E^{\tau }_3\setminus \partial E^{\tau }_3$ such that $x^1 \neq  x^2$ and $ \mathcal{U}(x^1)= \mathcal{U}(x^2) =z$. We can then choose small neighborhoods $ U_i$ of $x^i$ and $U$ of $z$ such that the restrictions $\mathcal{U}: U_i \to U$ are diffeomorphisms for each $i \in \{1,2\} $. This is a contradiction since $Y_4^{\tau } \cap U \neq \emptyset$, but there are at least two preimages of $\mathcal{U}$ for points in $Y_4^{\tau } \cap U$.
\end{proof}

\begin{lem}\label{Y^tau_4}
	There exists a uniform $L_0>0$ such that $\mathbb{R}^3\setminus B(0, L_0) \subset Y^{\tau }_4$.
\end{lem}
\begin{proof}
	Let's choose a uniform $L_0$ such that Lemma \ref{one-to-one-outer} holds. 
	From Lemma \ref{nonempty-preimage}, we know that $\mathcal{U}$ maps a neighborhood of the infinity in $E^{\tau }_3$ to a neighborhood of the infinity in $\mathbb{R}^3$. Additionally, according to Lemma \ref{E_2^tau}, we have $\mathcal{U}(\partial E^{\tau }_3) \subset B(0, L_0)$. Applying Lemma \ref{one-to-one-outer}, we can deduce that the mod $2$ degree of $\mathcal{U}: (E^{\tau }_3, \partial E^{\tau }_3) \to (\mathbb{R}^3, B(0, L_0) )$ is  one. By applying degree theory, we can conclude that $$\mathbb{R}^3\setminus B(0, L_0) \subset Y^{\tau }_3.$$
	Moreover, by Lemma \ref{one-to-one-outer}, it follows that $Y^{\tau }_4\cap ( \mathbb{R}^3\setminus B(0, L_0) ) = Y^{\tau }_3 \cap  (\mathbb{R}^3\setminus B(0, L_0) )$.

\end{proof}

\begin{prop}\label{regular-subregion}
	There exist uniform constants $C, L_0>0$ such that for any $\varepsilon \in (0, \frac{1}{100})$ and $\tau = m(g)^{\varepsilon }$, we can find a connected subregion $E^{\tau } \subset E^{\tau }_3$ satisfying
	$$
	M\setminus M_{L_0} \subset E^{\tau }$$
	and
	$$
	\mathcal{H}^2_g(\partial E^{\tau }) \leq C \cdot m(g)^{\frac{1}{2}- \varepsilon }.
	$$ 
	Moreover, the restriction map $$\mathcal{U}: E^{\tau }\to \mathbb{R}^3$$ is a diffeomorphism onto its image $Y^{\tau }:= \mathcal{U}(E^{\tau }) \subset \mathbb{R}^3$, and $\mathbb{R}^3\setminus B(0, L_0) \subset Y^{\tau }$. Such $E^{\tau }$ is called the regular subregion.
\end{prop}
\begin{proof}
	Notice that $\mathcal{U}(\partial E^{\tau }_3)$ is a smooth surface in $\mathbb{R}^3$, and from Lemma \ref{Y^tau open}, we have $\partial Y^{\tau }_4 \subset \mathcal{U}(\partial E^{\tau }_3)$. Thus, we can choose a smaller approximated region $Y^{\tau }_5 \subset Y^{\tau }_4$ such that $\partial Y^{\tau }_5$ is smooth,  $Y^{\tau }_5 \cup \partial Y^{\tau }_5 \subset Y^{\tau }_4$, $Y^{\tau }_4\setminus Y^{\tau }_5$ lies in a small neighborhood of $\mathcal{U}(\partial E^{\tau }_3)$, and $|\partial Y^{\tau }_5| \leq 2 |\mathcal{U}(\partial E^{\tau }_3)|$. Let $Y^{\tau }$ be the connected component of $Y^{\tau }_5$ such that $\mathbb{R}^3\setminus B(0, L_0) \subset Y^{\tau }$, and define $E^{\tau }:= \mathcal{U}^{-1}(Y^{\tau })$. From Lemma \ref{Y^tau_4}, $E^{\tau } \neq \emptyset$.

	By our definition, we know that $\mathcal{U}: E^{\tau }\to Y^{\tau }$ is a diffeomorphism and we have $M\setminus M_{L_0}\subset E^{\tau }$. Using the area formula and (\ref{E^tau_2}), we can deduce that
$$
\mathcal{H}^2_g (\partial E^{\tau }) = \int_{\partial Y^{\tau }}| J \mathcal{U}^{-1}| \leq (1+ 4 \tau ) |\partial Y^{\tau }| \leq 3 |\mathcal{U}(\partial E^{\tau }_2)| \leq 3 \int_{\partial E^{\tau }_2}| J \mathcal{U}| \leq 4 \mathcal{H}^2_g(\partial E^{\tau }_2),
$$ 
which together with Lemma \ref{E_2^tau} gives the conclusion. 
\end{proof}

In the following, our goal is to prove the convergence of the volumes of regular subregions as desired. To accomplish this, it is sufficient to prove it for a single large region, and a cylinder happens to be a convenient shape for this purpose.

For any fixed unit vector $a=(a_1,a_2,a_3) \in \mathbb{S}^2$, we define $u^a := \sum_{j=1}^{3}a_ju^j$. Now, let's define the coordinate cylinder $C_L^a:= D_L^{\pm} \cup T_L$, where
$$
D_L^{\pm}:= \{x \in M: u^a(x) = \pm L, |\mathcal{U}(x)|^2 - u^a(x) ^2 \leq L^2\} ,
$$ 
$$T_L:= \{x \in M: |u^a(x)| \leq L, |\mathcal{U}(x)|^2 - u^a(x)^2 = L^2\} .$$
Let $\Omega^a _L\subset M$ be the closure of the bounded component of $M\setminus C_L^a$. In the following, we will also use $\Omega _L^j$ to denote such a region in $(M, g)$ associated with $u^j$, and use $\Omega _{L, \mathrm{Eucl}}^j$ to denote such a region in $\mathbb{R}^3$ associated with $x^j$. Therefore, $\mathcal{U}(\Omega _L^j) \subset \Omega _{L, \mathrm{Eucl}}^j$.

\begin{lem}\label{volume-Omg_L}
	For a uniform $L_0$ as in Proposition \ref{regular-subregion}, and for any fixed $j\in \{1,2,3\} $ and $L> L_0$,
	$$
	\Vol_g(\Omega^j _L \cap E^{\tau } ) \geq (1-4\tau ) (1- \Psi (m(g) ) ) 2 \pi L^3 - \Psi (m(g)).
	$$ 
\end{lem}
\begin{proof}
Since $\Delta u^j=0$, we can apply integration by parts to obtain
	\begin{align*}
		\int_{\Omega ^j_L } |\nabla u^j|^2 &= \int_{T_L \cup D_L^{\pm} }u^j \left<\nabla u^j, \nu  \right>+ \int_{\partial M }u^j \left<\nabla u^j, \nu  \right>\\
					       &= (1\pm \Psi (m(g) ) ) 2\pi L^3,
	\end{align*}
	where we used Lemma \ref{onb-est} in the last equality.

Note that $M\setminus E^{\tau }$ is a bounded domain inside $M_{r_0}$ with boundary $\partial M \cup \partial E^{\tau }$, and $u^j|_{M_{r_0}} $ is uniformly bounded. By applying integration by parts, we obtain
	\begin{align*}
		\int_{M\setminus E^{\tau }}|\nabla u^j|^2& = \int_{\partial E^{\tau } }u^j \left<\nabla u^j, \nu  \right>  + \int_{\partial M}u^j \left<\nabla u^j, \nu  \right>\\
						       &= \int_{\partial E^{\tau }}u^j \left<\nabla u^j, \nu  \right>\\
						       &\leq C(r_0)(1+\tau ) \mathcal{H}_g^2(\partial E^{\tau } )\\
														   & \leq \Psi (m(g) ).
	\end{align*}
	So 
\begin{align*}
	(1+ \Psi (m(g) ) ) 2\pi L^3& \geq \int_{\Omega ^j_L \cap E^{\tau }}|\nabla u^j|^2\\
						  &\geq (1-\Psi (m(g) ) ) 2\pi L^3 - \Psi (m(g) ).
\end{align*} 
The conclusion follows from the fact that inside $ E^{\tau }$, we have $||\nabla u^j|^2 - 1| \leq 2\tau . $
\end{proof}

\begin{lem}\label{full-volume}
	For a uniform $L_0$ as in Proposition \ref{regular-subregion}, and for any $D_1>0, D_2 \geq L_0$, $\forall p^* \in B(0, D_1)$, we have
	$$|B(p^*, D_2) \setminus Y^{\tau }| \leq \Psi (m(g)|D_1, D_2).$$
\end{lem}
\begin{proof}
	Otherwise, assume that $|B(p^*, D_2)\setminus  Y^{\tau }|>c$ for some uniform $c >0$. Choose $L=2(D_1+D_2)$ so that $B(p^*, D_2) \subset \Omega _{L, \mathrm{Eucl}}^1$. Then we have
	\begin{align*}
		|\Omega _{L, \mathrm{Eucl}}^1 \cap Y^{\tau }| & \leq |\Omega _{L, \mathrm{Eucl}}^1 \setminus B(p^*,D_2)| + |B(p^*, D_2) \cap Y^{\tau }|\\
					   & \leq |\Omega _{L, \mathrm{Eucl}}^1 \setminus B(p^*,D_2)| + |B(p^*,D_2)| -c\\
					   & \leq |\Omega _{L, \mathrm{Eucl}}^1| -c\\
					   &= 2\pi L^3 -c.
	\end{align*}
	On the other hand, from Lemma \ref{volume-Omg_L} and (\ref{E^tau_2}), we know
	$$
	|\Omega _{L, \mathrm{Eucl}}^1 \cap Y^{\tau }| \geq (1- 4\tau )\Vol_g(\Omega _L^1 \cap E^{\tau }) \geq (1-4\tau )^2(1-\Psi (m(g) ) )2\pi L^3 - \Psi (m(g) ).
	$$ 
	Letting $m(g)\to 0$ leads a contradiction.
\end{proof}
\begin{rmk}
	Alternatively, one can prove this lemma by using Lemma \ref{vol-area} and the fact that $|\partial Y^{\tau }| \leq \Psi (m(g) )$. Such an argument will also be used later in the case of Ricci curvature bounded below.
\end{rmk}

Now we can prove the convergence theorem in the $C^0$ modulo negligible volume sense.

\begin{proof}[Proof of Theorem \ref{main-scalar}]
	Assume that $(M_i, g_i)$ is a sequence of $(A,B,\sigma )$-AF $3$-manifolds with nonnegative scalar curvature and their ADM masses $m(g_i)\to 0$. 

	By Proposition \ref{regular-subregion}, we know that for any $\varepsilon \in (0, \frac{1}{100})$ and $\tau _i= m(g_i)^{\varepsilon }$, the regular subregion $E^{\tau _i}$ is well defined and satisfies
	$$
	M_i\setminus M_{i, L_0} \subset E^{\tau _i}
	$$ for some uniform $L_0 >0$, and 
	$$
	\mathrm{Area}_{g_i}(\partial E^{\tau _i}) \leq C\cdot m(g_i)^{\frac{1}{2}- \varepsilon }.
	$$
	Moreover, there exist diffeomorphisms $\mathcal{U}_i=(u_i^1, u_i^2, u_i^3): E^{\tau _i}\to Y^{\tau _i}\subset \mathbb{R}^3$. We still denote $(\mathcal{U}_i)_* g_i$ by $g_i$. Then for any $y \in Y^{\tau _i}$, $g_{i}^{jk}(y)= \left<\nabla u_i^j, \nabla u_i^k \right>(\mathcal{U}_i^{-1}(y) )$ and by (\ref{E^tau_2}), we have
	\begin{align}\label{C^0-tau_i}
	|g_i^{jk}- \delta _{jk}|(y) \leq \tau _i \to 0.
\end{align} 

	Take $Z_i := M_i \setminus E^{\tau _i}$. It's remaining to show that for any base point $p_i \in  E^{\tau _i}$, $(E^{\tau _i}, g_i, p_i) \to (\mathbb{R}^3, g_E, 0)$ in the pointed $C^0$ modulo negligible volume sense.

	\textbf{Case 1)}: there exists a uniform $L>0$ such that $\mathcal{U}_i(p_i) \in  B(0, L)$. By composing $\mathcal{U}_i$ with a uniform translation of $\mathbb{R}^3$ which maps $\mathcal{U}_i(p_i)$ to $0$, we can assume that $\mathcal{U}_i(p_i) =0$. 

	Fix a radius $\rho >0$. For any $x \in \hat{B}(p_i, \rho ) \subset E^{\tau _i}$, and for any $\epsilon >0$, there exists a smooth curve $\gamma \subset E^{\tau _i}$ satisfying $\gamma (0)=p_i, \gamma (1)= x$ and $\mathrm{Length}_{g_i}(\gamma ) \leq \rho + \epsilon $. Let $\sigma := \mathcal{U}_i(\gamma )$. Then $\sigma (0)=0$, and by (\ref{C^0-tau_i}), 
	\begin{align}\label{|sgm|}
		|\sigma | \leq (1+ 4 \tau_i ) \mathrm{Length}_{g_i}(\gamma ) \leq (1+ 4 \tau _i)(\rho + \epsilon ).
	\end{align}
	Taking $\epsilon \to 0$, we have $\mathcal{U}_i(\hat{B}(p_i, \rho ) ) \subset B(0, (1+ 4\tau _i)\rho)$.

	Let $\Omega _i:= \mathcal{U}_i^{-1}(B(0, (1+4 \tau _i)\rho ) \cap Y^{\tau _i} )$ and $\Omega := B(0, \rho)$. Then $\Omega _i \supset \hat{B}(p_i, \rho )$. We claim that $\Omega _i \to \Omega $ in the $C^0$ modulo negligible volume sense as defined in Definition \ref{C^0 modulo}. 
To show this, we take $U_i = \Omega _i$ and $\varphi _i = \mathcal{U}_i$ there. By (\ref{C^0-tau_i}), we know $| (\mathcal{U}_i)_* g_i - g_E|_{C^0(\mathcal{U}_i(\Omega _i) )} \to 0$.  Applying Lemma \ref{full-volume}, we have 
	$$
	|\Omega \setminus \mathcal{U}_i(\Omega _i)| = |B(0, \rho ) \setminus Y^{\tau _i}| \leq \Psi (m(g_i)| \rho ) \to 0,
	$$ 
which implies the claim.

\textbf{Case 2)}: for any $L>0$, there exists a subsequence of $\mathcal{U}_i(p_i)$ such that $\mathcal{U}_i(p_i) \in \mathbb{R}^3\setminus B(0, L)$. Denote $\mathcal{U}_i(p_i)$ by $q_i$. By taking a further subsequence, we can assume that $|0q_i| \to \infty$. By Lemma \ref{nonempty-preimage}, there exists a sequence $L_i \to \infty$ such that $p_i \in M_i\setminus M_{i, L_i}$. Fix a radius $\rho >0$. Similar to (\ref{|sgm|}), we have 
\begin{align}\label{U_i-rho}
	\mathcal{U}_i(\hat{B}(p_i, \rho ) ) \subset B(q_i, (1+ 4\tau _i)\rho ).
\end{align}
Since $\partial Y^{\tau _i} \subset B(0, L_0)$, $B(q_i, (1+4 \tau _i)\rho) \subset Y^{\tau _i}\setminus \partial Y^{\tau _i}$ for all sufficiently large $i$. So for any $y \in B(q_i, (1+4 \tau _i)\rho ) $, the line segment $[q_i y] \subset Y^{\tau _i}$. Let $\gamma := \mathcal{U}_i^{-1}([q_i y])$. By (\ref{C^0-tau_i}), we have
$$
\mathrm{Length}_{g_i}(\gamma ) \leq (1+ 4 \tau _i)|q_i y| \leq (1+ 4\tau _i)^2 \rho ,
$$
which implies that $\mathcal{U}_i^{-1}(y) \in \hat{B}(p_i, (1+4 \tau _i)^2 \rho )$. Together with (\ref{U_i-rho}), we have
\begin{align}
	B(q_i, (1+4 \tau _i)^{-2} \rho ) \subset \mathcal{U}_i(\hat{B}(p_i, \rho ) ) \subset B(q_i, (1+4 \tau _i) \rho ).
\end{align}
By (\ref{C^0-tau_i}) again, we know that $(\hat{B}(p_i, \rho ), g_i)$ $C^0$-converges to $(B(0, \rho ), g_E)$ in the Cheeger-Gromov sense. 
\end{proof}

\section{Removable singularity under Ricci lower bound}\label{removable-sing}
In this section, we prove the Gromov-Hausdorff convergence when the Ricci curvature is uniformly bounded from below. Let's consider a sequence $(M_i,g_i,p_i)$ of $(A,B,\sigma )$-AF $3$-manifolds with nonnegative scalar curvature, $p_i \in M_i \setminus M_{i, L_0}$,  Ricci curvature bounded from below by $\Ric_{g_i}\geq -2\Lambda$, and the mass $m(g_i)\to 0$.

By Lemma \ref{precomp-ricci}, up to a subsequence, we can assume that for some complete length metric space $(X,d_X)$,

$$
(M_i, g_i, p_i)\to (X, d_X, p_X)
$$ 
in the pointed Gromov-Hausdorff topology. Since $p_i \in M_i\setminus M_{i, L_0}$, we know $\Vol_{g_i}B(p_i, 1) \geq v_0$ for some uniform $v_0>0$ and $\dim X =3$.

From Proposition \ref{regular-subregion}, we have the existence of regular subregions $E^{\tau _i}\subset M_{i, ext}$. Let's consider the restricted metric space $(E^{\tau _i}, d_i)$, where $d_i$ is the restriction of the distance induced by $(M_i, g_i)$. For any fixed radius $r>0$, by Lemma \ref{cpt-subset-converg}, up to a subsequence, we can assume that $(E^{\tau _i}\cap B(p_i, r), d_i)$ converges to a compact subset of $(X, d_X)$. By a diagonalization argument, we can find a closed subset $X' \subset X$ such that
$$
(E^{\tau _i}, d_i) \to (X', d_X).
$$ 
In other words, for any fixed radius $r>0$, $(E^{\tau _i} \cap B(p_i, r), d_i) \to (X' \cap B(p_X, r), d_X)$. 
Similarly, consider the restricted metric, and assume 
$$(\partial E^{\tau _i}, d_i) \to (\Sigma, d_X) \subset (X', d_X).$$
By Proposition \ref{regular-subregion}, $X'\setminus \Sigma \neq \emptyset$.

In general, $(X', d_X)$ is not a length space and is different from $(X,d_X)$. In our case, by using the fact that $\mathrm{Area}(\partial E^{\tau _i})\to 0$ and the conditon on the lower bound of Ricci curvature, we can prove the following.
\begin{prop}\label{X'=X}
	$X'=X$.
\end{prop}
\begin{proof}
	We argue by contradiction. Otherwise, take $y_0 \in X\setminus X' $. Choose $\delta _0>0$ small enough such that $d(y_0, X' ) > 2\delta _0>0$. Next, choose $x_0 \in X' \setminus \Sigma$, and by choosing $\delta _0$ smaller, we can assume $d(x_0, \Sigma)>2\delta _0$. Also, choose $D>0$ such that $x_0, y_0 \in B(p_X, \frac{D}{4})$. Assume $d(x_0, y_0) =L > 2\delta _0$ by choosing $\delta _0$ smaller. Now, take $x_i, y_i \in B(p_i, D)$ such that $x_i \in E^{\tau_i} $, $d_i(x_i, \partial E^\tau )>2\delta _0, d_i(y_i, \partial E^{\tau_i}) >2\delta _0$, and $x_i \to x_0, y_i \to y_0$ in the GH-topology. Here, one can consider that this GH-convergence takes place in a common metric space $X \sqcup (\sqcup_i M_i)$ with an admissible metric (see \cite[Section 11.1]{Petersen16} for more details). 

	By Proposition \ref{regular-subregion} and Lemma \ref{nbhd-point}, for sufficiently large $i$, there exists $z_i \in B(y_i, \delta _0)$ and a geodesic segment $\gamma _{x_i z_i}\subset (M_i\setminus \partial E^{\tau_i} , g_{i })$ between $x_i$ and $z_i$. Note that $M_{i,ext}$ is a connected component of $M_i \setminus \partial M_{i,ext}$ and $\gamma _{x_i z_i}\subset M_{i}\setminus \partial M_{i,ext}$, so $\gamma _{x_i z_i}\subset M_{i,ext}$ and $z_i \in M_{i,ext}$. 

	Also $\sum_{j,k=1}^3|\left<\nabla u_i^j, \nabla u_i^k \right> -\delta _{jk}|^2(z_i) < \tau_i $; otherwise, since $\sum_{j,k=1}^3|\left<\nabla u_i^j, \nabla u_i^k \right> -\delta _{jk}|^2(x_i) < \tau_i$, there will be an intersection point of $\gamma _{x_i z_i}$ and $\partial E^{\tau_i} $. Therefore, $z_i \in E^{\tau_i} $. 

	As we let $i\to \infty$, $z_i\to z_0 \in B(y_0, \delta _0) \cap X' $ in the GH-topology, which contradicts the assumption that $B(y_0, \delta _0) \cap X' = \emptyset$.
\end{proof}

From this proposition, we know that there exists $p_i' \in E^{\tau _i}$ such that $d_i(p_i, p_i') \to 0$. Notice that with any uniform perturbation of the base point, the pointed Gromov-Hausdorff limit space will remain unchanged. Therefore`, by applying a perturbation if necessary, we can assume that $p_i \in E^{\tau _i}$ for all sufficiently large $i$. Similar to case 1) in the proof of Theorem \ref{main-scalar},  we can also assume that $\mathcal{U}_i(p_i) \in B(0, L_1)$ for some uniform $L_1>L_0$. Additionally, through composition with uniform translation maps, we can ensure that $\mathcal{U}_i(p_i) = p^* $ for a fixed $p^* \in B(0, L_1) \setminus B(0, L_0)$.

Under diffeomorphisms $\mathcal{U}_i$, for the metrics $(\mathcal{U}_i^{-1})^* g_i$ over $\mathbb{R}^3\setminus B(0,L_0) \subset  Y^{\tau _i},$ which we continue to denote as $g_i$, we observe that
$$
(\mathbb{R}^3\setminus B(0,L_0), g_i) \to (\mathbb{R}^3\setminus B(0, L_0), g_E)
$$ locally in $C^0$-topology as tensors.  Furthermore, we want this convergence to hold in the pointed Gromov-Hausdorff topology, which is provided by the following lemma.

\begin{lem}\label{X'_delta}
	Fix $D, \delta >0$. For any $y, z \in E^{\tau _i} \cap B(p_i, D)$ with $d_i(y, \partial E^{\tau _i}) \geq \delta $ and $d_i(z, \partial E^{\tau _i}) \geq \delta$, we have
	$$
	| |\mathcal{U}_i(y) - \mathcal{U}_i(z)| - d_i(y,z)| \leq \Psi(m(g_i) | D, \delta ) .
	$$ 
\end{lem}
\begin{proof}
	For all large enough $i$, take $\delta _0= \Psi (m(g_i)| D, \delta )$ as in Lemma \ref{nbhd-point}. Then $\delta _0 <\frac{1}{2}\delta $, and by Proposition \ref{regular-subregion} and Lemma \ref{nbhd-point}, there exist $z'$ and a geodesic segment $\gamma _{y z'}$ connecting $y$ to $ z'$ such that $z' \in B(z, \delta _0)$ and $\gamma _{yz'}\cap \partial E^{\tau _i} = \emptyset$. In particular, $\gamma _{yz'}\subset E^{\tau _i}$. 
	
	Set $l= d_i(y, z')$. Note that
	$$|\mathcal{U}_i(y)- \mathcal{U}_i(z')|^2 =\sum_{j=1}^{3} (u_i^j(\gamma_{yz'} (l))- u_i^j(\gamma_{yz'} (0)) )^2= \sum_{j=1}^{3} \left( \int_{0}^{l} \left< \nabla u_i^j, \gamma_{yz'} ' \right> \right) ^2,$$
	and
$$\left| 1- \sum_{j=1}^{3}\left< \nabla u_i^j, \gamma_{yz'}' \right>^2 \right| < 4 \tau_i .$$ 
	So
	\begin{align}\label{yz'-ineq}
		\begin{split}
	|\mathcal{U}_i(y) - \mathcal{U}_i(z')| &\leq \left( \sum_{j=1}^3 l\int_{0}^{l} \left<\nabla u_i^j, \gamma_{yz'} ' \right>^2 \right) ^{\frac{1}{2}}\\
					      &\leq (1+ 4 \tau_i )l.
\end{split}
        \end{align}

	Since $d_i(z,z')<\frac{1}{2}\delta $ and $d_i(z, \partial E^{\tau _i}) >\delta $, there exists a geodesic segment connecting $z$ to $z'$ inside $E^{\tau _i}$, which implies that
	$$
	|\mathcal{U}_i(z)-\mathcal{U}_i(z')| \leq (1+4 \tau _i) d_i(z,z').
	$$ 
Thus 
$$
|\mathcal{U}_i(y) - \mathcal{U}_i(z)| \leq (1+ 4\tau _i)( d_i(y, z') + d_i(z, z') ) \leq (1+4 \tau _i)d_i(y, z) + 4 \delta _0,
$$ 
i.e. 
\begin{align}\label{yz-ineq}
|\mathcal{U}_i(y)-\mathcal{U}_i(z)|-d_i(y,z) \leq \Psi (m(g_i)|D, \delta ).
\end{align}

On the other hand, denote $y_i= \mathcal{U}_i(y), z_i=\mathcal{U}_i(z)$. Then choose a line segment $[y_iw]$ with $w \in \partial Y^{\tau _i}$ such that $|y_iw|= d_E(y_i, \partial Y^{\tau _i})$. Integration along $[y_iw]$ as above implies that 
$$
d_i(y, \mathcal{U}^{-1}_i(w) ) \leq (1+4 \tau _i) |y_i w|= (1+ 4 \tau _i) d_E(y_i, \partial Y^{\tau _i}).
$$
So $$
d_i(y, \partial E^{\tau _i}) \leq (1+ 4\tau _i) d_E(y_i, \partial Y^{\tau _i}),
$$ and particularly, $$d_E(y_i, \partial Y^{\tau _i})> \frac{1}{2}\delta. $$
Notice that the same inequality also holds for $z_i$.

By (\ref{yz-ineq}), we know that
$$
|y_iz_i| \leq d_i(y,z)+\Psi (m(g_i)|D,\delta ) \leq 2D + \Psi (m(g_i)|D,\delta ).
$$ 
By Lemma \ref{nbhd-point}, there exist $z_i' \in B(z_i, \delta _0)$ and a straight line segment $[y_iz_i']\subset Y^{\tau _i}$. Denote $z'= \mathcal{U}_i^{-1}(z_i') \in E^{\tau _i}$. Then $\gamma _1:=\mathcal{U}_i^{-1}([y_iz_i'])\subset E^{\tau _i}$ is a smooth curve between $y$ and $ z'$, $\gamma _2:=\mathcal{U}_i^{-1}([z_iz_i'])\subset E^{\tau _i}$ is a smooth curve between $z$ and $ z'$. Similar to (\ref{yz'-ineq}), we have
$$
d_i(y,z) \leq d_i(y,z') + d_i(z,z') \leq L_{g_i}(\gamma _1) + L_{g_i}(\gamma _2) \leq (1+ \Psi (m(g_i) ) ) (|y_iz_i| + \delta _0),
$$ 
i.e.
\begin{align}
d_i(y,z) - |\mathcal{U}_i(y)-\mathcal{U}_i(z)| \leq \Psi (m(g_i)|D, \delta ).
\end{align} 
\end{proof}

From this lemma, by taking a limit, we know that
\begin{align}\label{outer-isom}
(\mathbb{R}^3\setminus B(0,L_0), d_X)= (\mathbb{R}^3\setminus B(0,L_0), d_E) \subset (X, d_X).
\end{align}

Let $\{y^1, y^2, y^3\} $ be the coordinate functions of $\mathbb{R}^3$ and $e_j$ the unit vector along $y^j$-axis. Let $q_{\pm}^j=\pm 3L e_j$ for $L>4(L_0+L_1)$. By elementary geometry, we know that any $y\in B(0,3L) \subset \mathbb{R}^3$ is uniquely determined by $\{|y q_{\pm}^j|\} _{j=1}^3$. Particularly, there exists a continuous function $\xi : \mathbb{R}^6 \to \mathbb{R}$ with $\xi (0)=0$ such that for any $y, z \in B(0,3L)$, 
$$
|yz| = \xi ( |y q_{\pm}^j|-|z q_{\pm}^j|).
$$ 
We also choose an $\frac{1}{N}$-net $\{q^{a_k}\} _{k=1}^N$ of $A(0, L_0,3 L)$, where $A(0,L_0,3L)= B(0, 3L)\setminus B(0, L_0) \subset \mathbb{R}^3$ is the standard annulus. We always assume $\{q_{\pm}^j\} _{j=1}^3 \subset \{q^{a_k}\} _{k=1}^N$.
\begin{lem}\label{q-perturbation}
	For any fixed $N>10$, there exists $\delta _0= \Psi (m(g_i)|N ,L)$ such that for any $y \in Y^{\tau _i} \cap B(0, 2L)$, there exists $y' \in Y^{\tau _i}$ such that  $d_i(y, y')<\delta _0$ and
	$$
	|d_i(y', q^{a_k}) - |y' q^{a_k}| | \leq \Psi (m(g_i)|L ), \forall 1\leq k \leq N.
	$$ 
\end{lem}
\begin{proof}
	By the remark below Lemma \ref{nbhd-point}, there exists $\hat{B} \subset B_i(y, \delta _0) \subset M_i$ with $\Vol_i ( \hat{B}) \geq (1- \Psi (m(g_i)|N ,L) ) \Vol_i(B(y, \delta _0) )$ such that for any $y' \in \hat{B}$, the geodesic segment $\gamma _{q^{a_k}y'}$ lies in $Y^{\tau _i}$ for all $1 \leq k \leq N$. By volume comparison theorem, we know
	$$
	\frac{\Vol(B_i(y, \delta _0) )}{| B_{-\Lambda }(\delta _0)|} \geq \frac{\Vol (B_i(y, 2L) )}{| B_{-\Lambda }(2L)|} \geq \frac{\Vol( B_i(p_i, L) )}{|B_{-\Lambda }(2L)|} \geq \frac{|B(1)|}{|B_{-\Lambda }(2L)|}.
	$$ 
	So
	$$
	\frac{\Vol( B_i(y, \delta _0) )}{|B(\delta _0)|} \geq \frac{|B(1)|}{|B_{-\Lambda }(2L)|}\cdot \frac{|B_{-\Lambda }(\delta _0)|}{|B(\delta _0)|} \geq c(L)>0.
	$$ 
	Since $\hat{B}\subset Y^{\tau _i}$ and $|\hat{B}| \geq (1-\Psi ) \Vol_i(\hat{B})$, we have $$
	|\hat{B}| \geq (1-\Psi )c(L)|B(\delta _0)|.
	$$ 
	We claim that there exists $y' \in \hat{B}$ such that each straight line segment $[q^{a_k} y']$ lies in $Y^{\tau _i}$. Otherwise, by Lemma \ref{vol-area}, we have
	$$
	|\hat{B}| \leq C(N,L)\cdot |\partial Y^{\tau _i}|,
	$$ which leads to a contradiction if we choose $\delta _0= \Psi (m(g_i) |N, L)$ to be larger compared with $|\partial Y^{\tau _i}|$. 

	So, for any such $y' \in \hat{B}$, these geodesic segments $\gamma _{q^{a_k} y'}$ and straight lines $[q^{a_k} y']$ all lie in $Y^{\tau _i}$. Similar to the proof of Lemma \ref{X'_delta}, by taking integral along these segments, we can get the desired conclusion.
\end{proof}

Now we can define a map from $X$ to $\mathbb{R}^3$. For any large enough $L> 4(L_0+L_1)$ and $y \in  B(p_X, L)\subset X$, assume $x_i \in E^{\tau _i}\cap B(p_i, L)$ and $x_i \to y$ in the GH-topology. Denote $y_i= \mathcal{U}(x_i)$. Then, by Lemma \ref{X'_delta}, $y_i \in Y^{\tau _i}\cap B(p^*, 2L)$. For any fixed $N>10$, by Lemma \ref{q-perturbation}, there exists $y_{N,i}' \in Y^{\tau _i}$ with $d_i(y_i, y_{N,i}')< \Psi (m(g_i)|N, L )$ and 
$$
| d_i(y_{N,i}', q^{a_k})-|y_{N,i}'q^{a_k}| | \leq \Psi (m(g_i)|L ),\  \forall 1\leq k \leq N.
$$ 
Up to a subsequence, we can assume that $y_{N,i}'\to y_N' \in B(p^*,2L) \subset \mathbb{R}^3$ in the $d_E$-topology. For a sequence of $N \in \mathbb{N}$ approaching $\infty$, by considering a further subsequence, we can assume that $d_i(y_i, y_{N_i, i}')\to 0$ and for some $y' \in B(p^*, 2L)$, $y_{N_i,i}'\to y'$ in the $d_E$-topology.

Define
$$
\Xi _L : B(p_X,L) \to B(p^*, 2L) \subset  \mathbb{R}^3,
$$
by $\Xi_L (y)=y'.$ 

Notice that $\Xi _L(p_X) =p^* \in A(0, L_0, L)$. Additionally, from (\ref{outer-isom}), we know that $\Xi _L$ is the identity map when restricted to $A(0, L_0, 3L)$, and $\Xi _L^{-1}(A(0,L_0,3L) ) = A(0, L_0, 3L)$.

\begin{lem}\label{well-defined}
	$\Xi_L $ is well defined.
\end{lem}
\begin{proof}
	It's enough to show that for any $y \in B(p_X, L)$ with $y_i', y_i'' \in Y^{\tau _i}\cap B(p^*,2 L)$ satisfying $y_i', y_i'' \to y$ in the GH-topology and 
	$$
	| d_i(y_i', q_{\pm}^j)-|y_i'q_{\pm}^j| | \leq \Psi (m(g_i) ),\ | d_i(y_i'', q_{\pm}^j)-|y_i''q_{\pm}^j| | \leq \Psi (m(g_i) ), \forall j \in \{1,2,3\} ,
	$$ 
	if $y_i'\to y'$ and $y_i''\to y''$ in the $d_E$-topology, then $y'=y''$.

	To see this, by the assumptions, we know 
	$$
	d_X(y, q_{\pm}^j)=|y'q_{\pm}^j|=|y''q_{\pm}^j|,\ \forall j \in \{1,2,3\} .
	$$ 
	So $|y'y''|= \xi (|y'q_{\pm}^j|-|y''q_{\pm}^j|)=\xi (0)=0$.

\end{proof}

\begin{prop}\label{isometry}
	$\Xi_L: (B(p_X,L), d_X)\to (\mathbb{R}^3, d_E) $ is an isometry onto $B(p^*, L)$.
\end{prop}
\begin{proof}
	For any $y,z \in  B(p_X,L)$ and assume $y_i\to y, z_i\to z$ in the GH-topology with $y_i, z_i \in Y^{\tau _i} \cap B(p^*, 2L)$. Take $y_{N_i,i}', z_{N_i,i}'$ as mentioned above and assume $y_{N_i,i}'\to y', z_{N_i,i}'\to z'$ in the $d_E$-topology. Thus, $\Xi (y)=y'$ and $\Xi (z)=z'$. 

	Note that $d_X(y,z)=\lim_{i\to \infty}d_i(y_i,z_i)$ and 
\begin{align*}
	|d_i(y_{N_i,i}', z_{N_i,i}') - d_i(y_i, z_i)| \leq d_i(y_i, y_{N_i,i}') + d_i(z_i, z_{N_i,i}') \leq \Psi (m(g_i) ),
\end{align*}
	so $d_X(y,z)= \lim_{i\to \infty}d_i(y_{N_i,i}', z_{N_i,i}')$. 

	If we take $z= q^{a_k}$, $\forall 1 \leq k <\infty$, then
	$$
	d_X(y, q^{a_k}) = |y' q^{a_k}| .
	$$ 
	Since $\{q^{a_k}\} _{k=1}^\infty$ is dense in $A(0, L_0, 3L)$, for any $q^a \in A(0, L_0, 3L)$, we have
	\begin{align}\label{q^a-isom}
d_X(y, q^a) = |y'q^a|.
\end{align} 
In particular, $d_X(p_X, y)= |p^* y'|$.

For any $y,z \in  B(p_X, L)$, since $y', z' \in B(p^*,2 L)$, we can extend $[y'z']$ to a straight line segment in $B(p^*,2L) $ such that $q^a$ is one end point for some $q^a\in A(0, L_0, 3L)$, say $z' \in [q^ay']$. Then $$|y'z'|= |q^ay'|-|q^az'|= d_X(q^a, y)- d_X(q^a, z) \leq d_X(y,z).$$ 
So $\Xi _L$ is a $1$-Lipschitz map, and particularly continuous.

We claim that $\Xi _L$ is injective on $B(p_X, L)$. Otherwise, assume $\Xi _L(y_1)=\Xi _L(y_2)=y'$ with $y_1, y_2 \in B(p_X, L)$. For $j=1,2$, choose geodesic segments $\gamma _j$ between $y_j$ and $q^a$ for some $q^a\in A(0, L_0, 2L_0)$. Then $\gamma _j \subset B(p_X, L)$. Let $\sigma _j = \Xi _L(\gamma _j)$. Since $\Xi _L$ is $1$-Lipschitz, the length of $\sigma _j$ in $\mathbb{R}^3$ is smaller than the length of $\gamma _j$ in $X$, i.e. $$|y'q^a| \leq |\sigma _j| \leq d_X(y_j, q^a)= |y'q^a|,$$
where we used (\ref{q^a-isom}) for the last equality.
So $\sigma _1=\sigma _2=[y'q^a]$ is the unique line segment between $y'$ and $q^a$. Since $\Xi _L = \Id$ around $q^a$, we know $\gamma _1$ coincides with $\gamma _2$ around $q^a$. If $y_1 \neq y_2$, then we have branching geodesic segments in $X$. This is a contradiction with Proposition \ref{non-branching}.

By Theorem \ref{mfd-structure} and Lemma \ref{X'=X}, $B(p_X, L)$ is a topological manifold with boundary $\partial B(p_X, L)=\{x \in X: d(p_X, x)= L\} $. Since  $\Xi _L$ maps $\partial B(p_X, L)$ into $\partial B(p^*, L)$, and $\Xi _L= \Id$ on $A(0, L_0, L)$, it follows that for the induced map 
$$\Xi _L: (B(p_X, L), \partial B(p_X, L) ) \to (B(p^*, L), \partial B(p^*, L) ) ,$$
the mod $2$ degree is $1$. Applying the degree theory, we conclude that $$\Xi _L(B(p_X, L) ) = B(p^*, L).$$

So $\Xi _L: B(p_X, L) \to B(p^*, L)$ is a homeomorphism. It remains to show that it's also an isometric map. Take any $y,z \in B(p_X, L)$, and assume $y'=\Xi _L(y), z'=\Xi _L(z)$. Choose a line segment $[q^{a_-}q^{a_+}]$ with $q^{a_-}, q^{a_+}\in A(0,L_0,3L)$, length $l$ and $[y'z']\subset [q^{a_-}q^{a_+}]\subset B(p^*, L)$. Let $\gamma := \Xi _L^{-1}([q^{a_-}q^{a_+}])$. Then, due to the fact that  $\Xi _L$ is $1$-Lipschitz, we have $d_X( \gamma(t), \gamma (s) ) \geq |t-s| $ . Since $d_X( \gamma (0), \gamma (l) ) = d_X(q^{a_-}, q^{a_+}) = |q^{a_-}q^{a_+}|=l$ by (\ref{q^a-isom}), it must hold that 
$$
d_X(\gamma (t), \gamma (s) ) = |t-s|.
$$ 
In particular, $d_X(y, z) = |y'z'|$, which means $\Xi _L$ is an isometry. 

\end{proof}

Since $L$ could be any large enough positive number, this proposition implies that $(X, d_X)= (\mathbb{R}^3, d_E)$. This finishes the proof of Theorem \ref{main-ricci}.

\bibliographystyle{alpha}
\bibliography{./math}

\begin{thebibliography}{BKKS22}

\bibitem[ABK22]{ABK22}
Brian Allen, Edward Bryden, and Demetre Kazaras.
\newblock Stability of the positive mass theorem and torus rigidity theorems under integral curvature bounds.
\newblock {\em arXiv preprint arXiv:2210.04340}, 2022.

\bibitem[ADM61]{ADM61}
Richard Arnowitt, Stanley Deser, and Charles~W Misner.
\newblock Coordinate invariance and energy expressions in general relativity.
\newblock {\em Physical Review}, 122(3):997, 1961.

\bibitem[Bar86]{Bartnik86}
Robert Bartnik.
\newblock The mass of an asymptotically flat manifold.
\newblock {\em Communications on pure and applied mathematics}, 39(5):661--693, 1986.

\bibitem[BF02]{BF99}
Hubert Bray and Felix Finster.
\newblock Curvature estimates and the positive mass theorem.
\newblock {\em Comm. Anal. Geom.}, 10(2):291--306, 2002.

\bibitem[BKKS22]{BKKS22}
Hubert~L. Bray, Demetre~P. Kazaras, Marcus~A. Khuri, and Daniel~L. Stern.
\newblock Harmonic functions and the mass of 3-dimensional asymptotically flat {R}iemannian manifolds.
\newblock {\em The Journal of Geometric Analysis}, 32(6):1--29, 2022.

\bibitem[Bra01]{Bray01}
Hubert~L. Bray.
\newblock Proof of the {R}iemannian {P}enrose inequality using the positive mass theorem.
\newblock {\em Journal of Differential Geometry}, 59(2):177--267, 2001.

\bibitem[Cor05]{Cor05}
Justin Corvino.
\newblock A note on asymptotically flat metrics on $\mathbb{R}^3$ which are scalar-flat and admit minimal spheres.
\newblock {\em Proceedings of the American Mathematical Society}, 133(12):3669--3678, 2005.

\bibitem[Den20]{Deng20}
Qin Deng.
\newblock H\"older continuity of tangent cones in {$RCD(K,N)$} spaces and applications to non-branching.
\newblock {\em arXiv preprint arXiv:2009.07956}, 2020.

\bibitem[DS23]{DongSong23}
Conghan Dong and Antoine Song.
\newblock Stability of {E}uclidean 3-space for the positive mass theorem.
\newblock {\em arXiv preprint arXiv:2302.07414}, 2023.

\bibitem[FK02]{FK02}
Felix Finster and Ines Kath.
\newblock Curvature estimates in asymptotically flat manifolds of positive scalar curvature.
\newblock {\em Comm. Anal. Geom.}, 10(5):1017--1031, 2002.

\bibitem[Hei10]{Hein10}
Hans-Joachim Hein.
\newblock On gravitational instantons.
\newblock {\em ProQuest Dissertations and Theses}, pages 1--129, 2010.

\bibitem[HI01]{HuiskenIlmanen01}
Gerhard Huisken and Tom Ilmanen.
\newblock The inverse mean curvature flow and the {R}iemannian {P}enrose inequality.
\newblock {\em Journal of Differential Geometry}, 59(3):353--437, 2001.

\bibitem[HL15]{HL15}
Lan-Hsuan Huang and Dan~A Lee.
\newblock Stability of the positive mass theorem for graphical hypersurfaces of {E}uclidean space.
\newblock {\em Communications in Mathematical Physics}, 337:151--169, 2015.

\bibitem[KKL21]{KKL21}
Demetre Kazaras, Marcus Khuri, and Dan Lee.
\newblock Stability of the positive mass theorem under {R}icci curvature lower bounds.
\newblock {\em arXiv preprint arXiv:2111.05202}, 2021.

\bibitem[Lee09]{Lee09}
Dan~A Lee.
\newblock On the near-equality case of the positive mass theorem.
\newblock {\em Duke Math. J.}, 148(1):63--80, 2009.

\bibitem[Li18]{Li18}
Yu~Li.
\newblock Ricci flow on asymptotically {E}uclidean manifolds.
\newblock {\em Geometry \& Topology}, 22(3):1837--1891, 2018.

\bibitem[LNN20]{LNN20}
Man-Chun Lee, Aaron Naber, and Robin Neumayer.
\newblock $ d_p $ convergence and $\epsilon $-regularity theorems for entropy and scalar curvature lower bounds.
\newblock {\em arXiv preprint arXiv:2010.15663}, 2020.

\bibitem[LS14]{LS14}
Dan~A. Lee and Christina Sormani.
\newblock Stability of the positive mass theorem for rotationally symmetric {R}iemannian manifolds.
\newblock {\em Journal f{\"u}r die reine und angewandte Mathematik (Crelles Journal)}, 2014(686):187--220, 2014.

\bibitem[Pet16]{Petersen16}
Peter Petersen.
\newblock {\em Riemannian Geometry}.
\newblock Springer International Publishing AG, 3rd edition, 2016.

\bibitem[Ron06]{Rong06}
Xiaochun Rong.
\newblock Collapsed manifolds with bounded sectional curvature and applications.
\newblock {\em Surveys in differential geometry}, 11(1):1--24, 2006.

\bibitem[Sch89]{Schoen89}
Richard~M. Schoen.
\newblock Variational theory for the total scalar curvature functional for {R}iemannian metrics and related topics.
\newblock {\em Topics in calculus of variations}, pages 120--154, 1989.

\bibitem[Sor23]{Sormani23}
Christina Sormani.
\newblock Conjectures on convergence and scalar curvature.
\newblock In {\em Perspectives in Scalar Curvature}, pages 645--722. World Scientific, 2023.

\bibitem[ST21]{SimonTopping21}
Miles Simon and Peter~M. Topping.
\newblock Local mollification of {R}iemannian metrics using {R}icci flow, and {R}icci limit spaces.
\newblock {\em Geometry \& Topology}, 25(2):913--948, 2021.

\bibitem[Ste22]{Stern22}
Daniel~L. Stern.
\newblock Scalar curvature and harmonic maps to {$ S^{1}$}.
\newblock {\em Journal of Differential Geometry}, 122(2):259--269, 2022.

\bibitem[SY79a]{SchoenYau79b}
Richard Schoen and Shing-Tung Yau.
\newblock Complete manifolds with nonnegative scalar curvature and the positive action conjecture in general relativity.
\newblock {\em Proc. Nat. Acad. Sci.}, 76(3):1024--1025, 1979.

\bibitem[SY79b]{SchoenYau79a}
Richard Schoen and Shing-Tung Yau.
\newblock On the proof of the positive mass conjecture in general relativity.
\newblock {\em Communications in Mathematical Physics}, 65:45--76, 1979.

\bibitem[Wit81]{Witten81}
Edward Witten.
\newblock A new proof of the positive energy theorem.
\newblock {\em Communications in Mathematical Physics}, 80(3):381--402, 1981.

\end{thebibliography}

\end{document}